\documentclass[11pt]{article}

\usepackage[margin=1in]{geometry}
\usepackage{amsmath,amssymb,amsthm,mathtools}
\usepackage{enumitem}
\usepackage{comment}
\usepackage{microtype}
\usepackage{hyperref}
\usepackage{xcolor}
\usepackage{authblk}
\usepackage{graphicx}
\usepackage{cleveref}
\usepackage[backend=biber, style=numeric, maxnames=99]{biblatex} 
\hypersetup{
  colorlinks=true,
  linkcolor=red,
  urlcolor=red,
  citecolor=red
}

\newcommand{\KS}{\mathrm{KS}}
\newcommand{\KL}{\mathrm{KL}}
\newcommand{\MMD}{\mathrm{MMD}}
\newcommand{\KLinf}{\mathrm{KL}_{\mathrm{inf}}}

\newcommand{\dinfKL}{d_{\mathrm{infKL}}}

\newcommand{\GRO}{\operatorname{GRO}}

\newcommand{\REGROW}{\operatorname{REGROW}}
\newcommand{\E}{\mathbb{E}}
\newcommand{\RR}{\mathbb{R}}
\newcommand{\Pcal}{\mathcal{P}}
\newcommand{\Qcal}{\mathcal{Q}}

\newcommand{\Fcal}{\mathcal{F}}
\newcommand{\X}{\mathsf{X}}
\newcommand{\M}{\mathcal{M}}
\newcommand{\TV}{\mathrm{TV}}
\newcommand{\Wass}{\mathsf W}
\newcommand{\dd}{d}
\newcommand{\B}{\mathcal{B}}
\newcommand{\one}{\mathbf{1}}

\newtheorem{theorem}{Theorem}[section]
\newtheorem{proposition}[theorem]{Proposition}
\newtheorem{lemma}[theorem]{Lemma}
\newtheorem{assumption}[theorem]{Assumption}
\newtheorem{corollary}[theorem]{Corollary}
\theoremstyle{definition}
\newtheorem{definition}[theorem]{Definition}
\newtheorem{remark}[theorem]{Remark}

\theoremstyle{definition}

\author[]{Ashwin Ram\thanks{aram2@andrew.cmu.edu} }
\author[]{Aaditya Ramdas\thanks{aramdas@cmu.edu}}
\affil[]{Carnegie Mellon University}

\title{Power-one sequential tests for \texorpdfstring{$\Pcal$ against $\Pcal^c$}{P against P-complement}}
\date{\today}

\begin{document}

\maketitle

\begin{abstract}
We study power-one sequential testing for an i.i.d. law on a Polish sample space.  Given a nonempty composite null class $\Pcal\subseteq\mathcal M_1(\X)$, we ask when there exists a level-$\alpha$ stopping rule that rejects almost surely under every alternative in a prescribed class $\Qcal\subseteq\Pcal^c$.  Our main sufficient condition is local weak lower semicontinuity and positivity of the information projection functional
\(
\Phi_\Pcal(Q):=\inf_{P\in\Pcal}\KL(Q\|P).
\)
In particular, if $\Pcal$ is weakly compact, then for every $\alpha\in(0,1)$ there is a single level-$\alpha$ sequential test with power one against the entire complement $\Pcal^c$.  The proof combines Csisz\'ar's nonasymptotic Sanov bound for weakly closed convex empirical-measure sets with a Lindel\"of countable-subcover argument.  We also show that weak lower semicontinuity is sufficient but not necessary by giving examples where discontinuous finite-sample events separate alternatives that weak neighborhoods cannot detect.
Finally, we construct an $e$-process that is asymptotically relatively growth-rate optimal under weak compactness.  
We verify the weak-lower-semicontinuity condition for weakly compact nulls, $f$-divergence balls, several integral probability metric balls, Wasserstein balls on proper spaces, and a number of non-weakly-compact semiparametric examples.
\end{abstract}

\section{Introduction}
Sequential tests allow the sample size to be chosen from the data while maintaining a prescribed type-I error probability.  This feature is essential in modern anytime-valid inference: evidence can be monitored continuously, experiments can be stopped early, and validity is retained under optional stopping and optional continuation.  Classical work of Wald initiated the theory for simple hypotheses \parencite{wald1945}, and work of Robbins, Siegmund, Lai and collaborators developed power-one sequential tests and confidence sequences in several nonparametric  settings \parencite{darlingrobbins1967, darlingrobbins1967iterated, darlingrobbins1968, robbins1970, robbins_siegmund_1970, robbins_siegmund_1974, lai1976, lai1977}.

This paper considers a general i.i.d. composite testing problem.  The sample space $\X$ is Polish, the null is an arbitrary nonempty class $\Pcal\subseteq\mathcal M_1(\X)$, and the alternative is a class $\Qcal\subseteq\Pcal^c$.  A level-$\alpha$ sequential test is a stopping time $\tau$ satisfying
\[
  \sup_{P\in\Pcal}P^\infty(\tau<\infty)\le \alpha,
\]
and it has power one against $Q$ if $Q^\infty(\tau<\infty)=1$.  The central question is: is there a simple sufficient condition under which a single level-$\alpha$ stopping rule have power one for all $Q\in\Qcal$?

Our first answer is a weak-topological sufficient condition.  Define
\[
  \Phi_\Pcal(R):=\inf_{P\in\Pcal}\KL(R\|P).
\]
If $\Phi_\Pcal(Q)>0$ and $\Phi_\Pcal$ is weakly lower semicontinuous at each $Q\in\Qcal$, then for every $\alpha\in(0,1)$ there exists a level-$\alpha$ sequential test with power one against every $Q\in\Qcal$.  In particular, if $\Pcal$ is weakly compact, then $\Phi_\Pcal$ is weakly lower semicontinuous everywhere and $\Phi_\Pcal(Q)>0$ for every $Q\notin\Pcal$; hence weak compactness of the null implies power-one testability against the full complement $\Pcal^c$.

The proof uses only the weak topology and relative entropy.  For each alternative $Q$, lower semicontinuity gives a weak neighborhood on which $\Phi_\Pcal$ is bounded away from zero.  A closed convex bounded-Lipschitz ball inside this neighborhood has exponentially small empirical-measure probability under every null distribution by Csisz\'ar's nonasymptotic Sanov theorem.  The empirical measure enters this ball eventually under $Q$ by the law of large numbers.  For composite alternatives, the weak topology on $\mathcal M_1(\X)$ is second countable, so a Lindel\"of argument reduces the uncountable family of local neighborhoods to a countable subcover; allocating error probabilities over this subcover gives one valid stopping rule.

Weak lower semicontinuity is not necessary.  It is a convenient sufficient condition that detects alternatives through weak neighborhoods of empirical measures, but some alternatives are separated by discontinuous finite-sample events that weak neighborhoods cannot detect.  The example in Section~\ref{sec:notnec} illustrates this distinction and shows that the weak-topological condition should not be interpreted as a necessary condition.

We also discuss growth-rate optimality.  For a point alternative $Q$, the raw quantity $\Phi_\Pcal(Q)$ is not the correct optimal betting-wealth growth benchmark in full generality.  A recent work \cite{ramramdas2026optimalbettingwealth} shows that the optimal asymptotic growth rate is the bipolar inf-KL rate
\[
  \dinfKL(Q;\Pcal)
  :=\lim_{n\to\infty}\frac1n
  \inf_{R\in(\Pcal^n)^{\circ\circ}}\KL(Q^n\|R),
  \qquad
  \Pcal^n:=\{P^n:P\in\Pcal\},
\]
where $(\Pcal^n)^{\circ\circ}$ is the bipolar, or effective $n$-sample null, generated by all valid nonnegative $n$-sample tests \parencite{ramramdas2026optimalbettingwealth}.  This rate can be strictly smaller than the raw KL infimum.  However, the same work shows that the two quantities coincide when $\Phi_\Pcal$ is weakly lower semicontinuous at $Q$; weak compactness is an important sufficient condition.  Therefore the REGROW theorem in this paper should be read as follows: under weak compactness, the universal $e$-process we construct achieves the true bipolar inf-KL benchmark, and this benchmark equals the simpler expression $\Phi_\Pcal(Q)$.

The paper is organized as follows.  Section~\ref{sec:warm} gives notation, polars, bipolars, and entropy preliminaries.  Sections~\ref{sec:simple} and~\ref{sec:composite} prove the local-lower-semicontinuity sufficient condition for simple and composite alternatives.  Section~\ref{sec:notnec} shows that weak lower semicontinuity is not necessary.  Section~\ref{sec:eproc} aggregates power-one tests into a divergent $e$-process.  Section~\ref{sec:regrow} proves the corrected REGROW result under weak compactness.  Section~\ref{sec:examples} and the appendices verify the lower-semicontinuity condition in several statistically relevant examples.
\section{Preliminaries}\label{sec:warm}
Throughout, we let $(\X, \B)$ be a Polish space with its Borel $\sigma$-field, $\B$. We let $\mathcal{M}_1(\X)$ be the set of Borel probability measures on $(\X,\B)$. In addition, more generally for a measurable space $(\Omega, \mathcal F)$, we take  $\mathcal{M}_1(\Omega)$ to be the set of probability measures on $(\Omega, \mathcal F)$. Weak convergence of measures is denoted by $\Rightarrow$. Given all this, we will first define the KL divergence.

\begin{definition}\label{def:kl}
For $M, N\in\mathcal{M}_1(\X)$ we denote the KL divergence as,
\[
\KL(M\|N):=
\begin{cases}
\displaystyle \int_\X \log\left(\frac{dM}{dN}\right)dM, &\text{if } M\ll N,\\[0.6em]
+\infty, &\text{otherwise.}   
\end{cases}
\]
\end{definition}

\noindent Importantly, we do not use any dominating reference measure in our above definition of the KL. Only the Radon-Nikodym derivative when it exists.

For a nonempty null $\Pcal\subseteq\mathcal M_1(\X)$ define the raw KL-infimum functional
\[
  \Phi_\Pcal(R):=\inf_{P\in\Pcal}\KL(R\|P),
\]
with the convention that the infimum of an empty set is not used because $\Pcal$ is always assumed nonempty.  We say that $\Phi_\Pcal$ is \emph{weakly lower semicontinuous at $Q$} if $R_k\Rightarrow Q$ implies
\[
  \Phi_\Pcal(Q)\le \liminf_{k\to\infty}\Phi_\Pcal(R_k).
\]

Weak compactness of $\Pcal$ is one useful way to obtain lower semicontinuity of $\Phi_\Pcal$; see Lemma~\ref{lem:phi-lsc} below.  We first record the basic joint lower semicontinuity of relative entropy.

\begin{lemma}\label{lem:kl-lsc}
The map $(M, N)\mapsto\KL(M\|N)$ from $\mathcal{M}_1(\X)\times\mathcal{M}_1(\X)$ to $[0,\infty]$ is lower semicontinuous for the weak topology on both coordinates.  
\end{lemma}

\begin{proof}
Use the variational formula
\[
  \KL(Q\|P)=\sup_{g\in C_b(\X)}\left\{\int gd Q-\log\int e^gd P\right\},
\]
where $C_b(\X)$ is the set of bounded continuous functions on $\X$.
For each fixed bounded continuous \(g\), both integrals are continuous under weak convergence.  Taking the supremum of continuous lower bounds gives the claim.
\end{proof}

\section{Power-One Sequential Tests Against Simple Alternatives}\label{sec:simple}

We now analyze some applications to testing from all that we have shown so far. To begin, let $(X_i)_{i\ge 1}$ be the coordinate maps on $(\X^{\mathbb N},\B^{\otimes \mathbb N})$. And, let $\mathcal F_n:=\sigma(X_1, \dots, X_n)$ just be the natural filtration. For $P\in\mathcal M_1(\X)$, we let $P^\infty:=P^{\otimes \mathbb N}$ for the iid law on $\X^{\mathbb N}$. With this setup in mind, let us now define both sequential level and power.

\begin{definition}\label{def:seq-test}
Take a particular null class $\Pcal\subseteq\mathcal M_1(\X)$ and an alternative $Q\in\mathcal M_1(\X)$. A sequential test is simply a $(\mathcal F_n)$-stopping rule $\tau:\X^{\mathbb N}\to\mathbb N\cup\{\infty\}$. Now for some particular $\alpha\in(0,1)$ we say that $\tau$ has level at most $\alpha$ uniformly over $\Pcal$ if $\sup_{P\in\Pcal} P^\infty(\tau<\infty)\le \alpha$. We say that $\tau$ has power one against $Q$ if $Q^\infty(\tau<\infty)=1$.   
\end{definition}

\noindent In addition, for our statements below, we will need Pinsker's Inequality. Recall that in our setup, Pinsker tells us that for any probability measures $\mu, \nu$ on a measurable space,
\begin{equation}\label{eq:pink}
\|\mu-\nu\|_{\mathrm{TV}}\le \sqrt{\frac12\KL(\mu\|\nu)},
\end{equation}
where the rhs will be $+\infty$ if $\KL(\mu\|\nu)=\infty$.

\noindent With all this presented, we will now formally present a proposition to establish something very closely related to the following statement: informally, a test is power one if and only if $\KLinf>0$.

\begin{proposition}\label{prop:seq-power-one}
Fix $\alpha\in(0,1)$, $Q\in\mathcal M_1(\X)$, and a nonempty $\Pcal\subseteq\mathcal M_1(\X)$.  Let
\[
  d:=\Phi_\Pcal(Q)=\inf_{P\in\Pcal}\KL(Q\|P)\in[0,\infty].
\]
Then the following hold.
\begin{enumerate}
    \item If $d=0$, then no sequential test of level at most $\alpha$ can have power one against $Q$.
    \item If $d>0$ and $\Phi_\Pcal$ is weakly lower semicontinuous at $Q$, then there exists a sequential test of level at most $\alpha$ with power one against $Q$.
\end{enumerate}
Consequently, under the lower-semicontinuity assumption in Theorem~\ref{thm:mainy}, a level-$\alpha$ power-one sequential test against the point alternative $Q$ exists if and only if $\Phi_\Pcal(Q)>0$.
\end{proposition}

The proof relies on the following nonasymptotic version of Sanov's theorem. Note that our lemma's hypothesis also entails that Csisz{\'a}r's assumption of almost complete convexity holds.

\begin{lemma}[Csisz{\'a}r~\cite{csiszar1984sanov}]\label{lem:convex-empirical}
Let $P\in\mathcal M_1(\X)$ and let $C\subseteq\mathcal M_1(\X)$ be convex and weakly closed. Define $A_n:=\{\widehat Q_n\in C\}\subseteq \X^n$. Then it follows that for every $n\ge 1$,
\[
P^n(A_n)\ \le\ \exp\Bigl(-n\inf_{R\in C}\KL(R\|P)\Bigr).
\]
\end{lemma}

\begin{proof}[Proof of Proposition~\ref{prop:seq-power-one}]
Assume first that $d=0$.  For each $n\ge1$, choose $P_n\in\Pcal$ such that $\KL(Q\|P_n)\le n^{-2}$.  Let $\tau$ be any stopping rule satisfying $\sup_{P\in\Pcal}P^\infty(\tau<\infty)\le\alpha$, and set $B_n:=\{\tau\le n\}\in\mathcal F_n$.  Then
\[
  P_n^n(B_n)=P_n^\infty(\tau\le n)\le P_n^\infty(\tau<\infty)\le\alpha.
\]
Moreover,
\[
  \KL(Q^n\|P_n^n)=n\KL(Q\|P_n)\le n^{-1}.
\]
Pinsker's inequality gives $\|Q^n-P_n^n\|_{\mathrm{TV}}\le (2n)^{-1/2}$, and hence
\[
  Q^\infty(\tau\le n)=Q^n(B_n)
  \le P_n^n(B_n)+\|Q^n-P_n^n\|_{\mathrm{TV}}
  \le \alpha+\sqrt{\frac1{2n}}.
\]
Since $\{\tau\le n\}\uparrow\{\tau<\infty\}$, continuity from below yields
\[
  Q^\infty(\tau<\infty)
  =\lim_{n\to\infty}Q^\infty(\tau\le n)
  \le\alpha.
\]
Thus no level-$\alpha$ test can have power one against $Q$.

Now assume $d>0$ and $\Phi_\Pcal$ is weakly lower semicontinuous at $Q$.  Choose a finite number $r$ with $0<r<d$; this covers both finite and infinite $d$.  By lower semicontinuity, the set
\[
  O_r:=\{R\in\mathcal M_1(\X):\Phi_\Pcal(R)>r\}
\]
is a weak neighborhood of $Q$.  Let $d_{\mathrm{BL}}$ be a bounded-Lipschitz metric that generates the weak topology, and choose $\delta>0$ such that the closed ball
\[
  C:=\overline B^{\mathrm{BL}}_\delta(Q)
\]
is contained in $O_r$.  The set $C$ is weakly closed and convex.  Define $A_n:=\{\widehat Q_n\in C\}\subseteq\X^n$ and $\beta_n:=\sup_{P\in\Pcal}P^n(A_n)$.  Lemma~\ref{lem:convex-empirical} gives, for every $P\in\Pcal$,
\[
  P^n(A_n)
  \le \exp\left(-n\inf_{R\in C}\KL(R\|P)\right).
\]
Taking the supremum over $P\in\Pcal$ and using $C\subseteq O_r$ yields
\[
  \beta_n\le \exp\left(-n\inf_{R\in C}\Phi_\Pcal(R)\right)\le e^{-nr}.
\]
Choose $N$ such that $\sum_{n\ge N}e^{-nr}\le\alpha$, and define
\[
  \tau:=\inf\{n\ge N:\widehat Q_n\in C\}.
\]
Under $Q^\infty$, the empirical measures $\widehat Q_n\Rightarrow Q$ almost surely, so $\widehat Q_n\in C$ eventually and $Q^\infty(\tau<\infty)=1$.  Under any $P\in\Pcal$,
\[
  P^\infty(\tau<\infty)
  \le \sum_{n\ge N}P^n(A_n)
  \le \sum_{n\ge N}\beta_n
  \le\alpha.
\]
This proves the existence of a level-$\alpha$ sequential test with power one against $Q$.
\end{proof}

\section{Power-One Sequential Tests against Composite \texorpdfstring{$\mathcal Q$}{Q}}\label{sec:composite}
We next pass from a single alternative to a possibly uncountable alternative class.  Let $\Pcal\subseteq\mathcal M_1(\X)$ be a nonempty null and let $\Qcal\subseteq\mathcal M_1(\X)$ be an alternative class.  A stopping time $\tau$ has power one against $\Qcal$ if $Q^\infty(\tau<\infty)=1$ for every $Q\in\Qcal$, and it has level at most $\alpha$ uniformly over $\Pcal$ if $\sup_{P\in\Pcal}P^\infty(\tau<\infty)\le\alpha$.  The first observation is that an alternative with zero raw KL-infimum cannot be included in any power-one class.

\begin{proposition}\label{prop:composite-necessary}
If there exists some $Q_0\in\Qcal$ with $\KLinf(Q_0, \Pcal)=0$, then for every single $\alpha\in(0,1)$ and every stopping time $\tau$, we have that $\sup_{P\in\Pcal}P^\infty(\tau<\infty)\le \alpha$ implies that $Q_0^\infty(\tau<\infty)\le \alpha$. As such, no level-$\alpha<1$ sequential test can have power one simultaneously against every $Q\in\Qcal$.   
\end{proposition}

\begin{proof}
This is Proposition~\ref{prop:seq-power-one} applied to the single alternative $Q_0$.
\end{proof}

Thus pointwise positivity of $\Phi_\Pcal$ on $\Qcal$ is necessary for power-one testing.  The following theorem gives a sufficient condition that does not require any bound on the cardinality of $\Qcal$ or any uniform KL separation.  The only topological input is that, because $\X$ is Polish, $\mathcal M_1(\X)$ with the weak topology is a separable metrizable space and hence Lindel{\"o}f \parencite{Good1995,steen1978counterexamples}.  We isolate this countability reduction first.

\begin{lemma}\label{lem:lindelof}
Every open cover of $\Qcal\subseteq\mathcal M_1(\X)$ in the weak topology has a countable subcover.    
\end{lemma}

\begin{proof}
Since $\X$ is Polish, the weak topology on $\mathcal M_1(\X)$ is metrizable and separable.  Every separable metric space is second countable, every subspace of a second-countable space is second countable, and every second-countable space is Lindel{\"o}f.  Hence every open cover of $\Qcal$ has a countable subcover.
\end{proof}

We now construct the uniform power-one test from a countable cover of KL-separated neighborhoods.

\begin{theorem}\label{thm:mainy}
Assume that the following hold.
\begin{enumerate}
    \item For every $Q\in\Qcal$, $\Phi_\Pcal(Q)>0$; the value $+\infty$ is allowed.
    \item For every $Q\in\Qcal$, the map $R\mapsto \Phi_\Pcal(R)$ is weakly lower semicontinuous at $Q$.
\end{enumerate}
Then, for every $\alpha\in(0,1)$, there exists a single stopping rule $\tau$ such that
\[
  \sup_{P\in\Pcal}P^\infty(\tau<\infty)\le\alpha
  \qquad\text{and}\qquad
  Q^\infty(\tau<\infty)=1\quad\text{for all }Q\in\Qcal.
\]
\end{theorem}

\begin{proof}
Fix $\alpha\in(0,1)$ and choose positive numbers $(\alpha_j)_{j\ge1}$ with $\sum_{j\ge1}\alpha_j\le\alpha$, for instance $\alpha_j=\alpha2^{-j}$.  Let $d_{\mathrm{BL}}$ be a bounded-Lipschitz metric that generates the weak topology on $\mathcal M_1(\X)$.

For each $Q\in\Qcal$, choose a finite $r_Q$ with $0<r_Q<\Phi_\Pcal(Q)$.  Since $\Phi_\Pcal$ is weakly lower semicontinuous at $Q$, the set
\[
  V_Q:=\{R:\Phi_\Pcal(R)>r_Q\}
\]
is a weak neighborhood of $Q$.  Choose $\delta_Q>0$ such that the closed bounded-Lipschitz ball
\[
  \overline U_Q:=\overline B^{\mathrm{BL}}_{\delta_Q}(Q)
\]
is contained in $V_Q$, and write $U_Q:=B^{\mathrm{BL}}_{\delta_Q}(Q)$.  The set $\overline U_Q$ is weakly closed and convex, and
\[
  \inf_{R\in\overline U_Q}\Phi_\Pcal(R)\ge r_Q>0.
\]
The family $\{U_Q:Q\in\Qcal\}$ is an open cover of $\Qcal$.  By Lemma~\ref{lem:lindelof}, it has a countable subcover, which we denote by $\{U_j:j\ge1\}$.  Write $\overline U_j$ and $r_j$ for the corresponding closed ball and rate.

For each $j,n$, define
\[
  A_{j,n}:=\{\widehat Q_n\in U_j\},
  \qquad
  \beta_{j,n}:=\sup_{P\in\Pcal}P^n(A_{j,n}).
\]
Because $A_{j,n}\subseteq\{\widehat Q_n\in\overline U_j\}$ and $\overline U_j$ is weakly closed and convex, Lemma~\ref{lem:convex-empirical} gives
\[
  P^n(A_{j,n})
  \le P^n(\widehat Q_n\in\overline U_j)
  \le \exp\left(-n\inf_{R\in\overline U_j}\KL(R\|P)\right)
\]
for every $P\in\Pcal$.  Therefore
\[
  \beta_{j,n}
  \le \exp\left(-n\inf_{R\in\overline U_j}\Phi_\Pcal(R)\right)
  \le e^{-nr_j}.
\]
Choose $N_j$ large enough that $\sum_{n\ge N_j}e^{-nr_j}\le\alpha_j$, and define
\[
  \tau_j:=\inf\{n\ge N_j:\widehat Q_n\in U_j\}.
\]
Then, for every $P\in\Pcal$,
\[
  P^\infty(\tau_j<\infty)
  \le \sum_{n\ge N_j}P^n(A_{j,n})
  \le \alpha_j.
\]
If $Q\in U_j$, then $\widehat Q_n\Rightarrow Q$ almost surely under $Q^\infty$, and hence $\widehat Q_n\in U_j$ eventually.  Thus $Q^\infty(\tau_j<\infty)=1$ for every $Q\in U_j$.

Finally set $\tau:=\inf_{j\ge1}\tau_j$.  For any $P\in\Pcal$,
\[
  P^\infty(\tau<\infty)
  \le \sum_{j\ge1}P^\infty(\tau_j<\infty)
  \le \sum_{j\ge1}\alpha_j
  \le\alpha.
\]
For any $Q\in\Qcal$, choose $j$ such that $Q\in U_j$.  Then $Q^\infty(\tau_j<\infty)=1$ and $\tau\le\tau_j$, so $Q^\infty(\tau<\infty)=1$.
\end{proof}

\noindent 
\section{Weak lower semicontinuity not necessary for power-one testing}\label{sec:notnec}

The next example shows that weak lower semicontinuity is not necessary for power-one testing.  The alternative is separated from the null by a discontinuous finite-sample event even though weak neighborhoods do not detect the separation.

\begin{proposition}\label{prop:weaklsc-not-necessary}
Let $\X=[0,1]$ on the Borel $\sigma$-algebra. Let $Q=\delta_0$ and, for $k\in\mathbb N$, define $P_k:=\frac12\delta_0 +\frac12\delta_{1/k}$. Set $\Pcal:=\{P_k:k\in\mathbb N\}$. Then
\begin{enumerate}
    \item $\KLinf(Q,\Pcal)=\log 2>0$.
    \item $\Phi(R)=\inf_{P\in\Pcal}\KL(R\|P)$ is not weakly lower semicontinuous at $Q$.
    \item For every $\alpha\in(0,1)$, there exists a level $\alpha$ sequential test with power-one against $Q$.
    \item There exists an e-process $(E_n)_{n\ge 0}$ for $\Pcal$ with exact growth rate $\log 2$ under $Q$.
\end{enumerate}
\end{proposition}

\begin{proof}
For the first claim, fix $k$. Since $P_k(\{0\})=1/2$, hence $Q=\delta_0\ll P_k$. Therefore
\[
\KL(Q\|P_k)=\int\log\Bigl(\frac{dQ}{dP_k}\Bigr)dQ =\log\Bigl(\frac{Q(\{0\})}{P_k(\{0\})}\Bigr) =\log\Bigl(\frac{1}{1/2}\Bigr)=\log 2.
\]
Taking the infimum over $k$ gives $\KLinf(Q,\Pcal)=\log 2$. For the second claim, we have that $P_k\Rightarrow Q$ weakly since for every $f\in C_b([0,1])$,
\[
\int fdP_k=\tfrac12 f(0)+\tfrac12 f(1/k)\to \tfrac12 f(0)+\tfrac12 f(0)=\int fdQ.
\]
For each $k$, $\Phi(P_k)=0$ because $\KL(P_k\|P_k)=0$. As shown above, $\Phi(Q)=\log 2>0$. Therefore $\Phi(Q)=\log 2\not\le \liminf_{k\to\infty}\Phi(P_k)=0$. Thus $\Phi$ is not weakly lower semicontinuous at $Q$. For the third claim, consider some $\alpha\in(0,1)$ and choose $m$ with $2^{-m}\le\alpha$. Let $A_m:=\{(0,\ldots,0)\}\subseteq[0,1]^m$ and define
\[
\tau:=
\begin{cases}m,&\text{if }(X_1,\dots,X_m)\in A_m,\\
\infty,&\text{otherwise.}
\end{cases}
\]
Under $Q^\infty$ we have that $X_i=0$ with probability one for all $i$, so $\tau=m$ almost surely and the power is one. Note that under any $P_k^\infty$, we have $P_k^\infty(\tau<\infty)=P_k^m(A_m)=2^{-m}\le \alpha$. It follows that the level is at most $\alpha$ uniformly over $\Pcal$. For the fourth claim, define $E_0:=1$ and for $n\ge 1$,
\[
E_n:=2^n\one\{X_1=\cdots=X_n=0\}.
\]
For each $k$ and each $n\ge 1$, we have that,
\begin{align*}
\E_{P_k^\infty}[E_n\mid\mathcal F_{n-1}]&=\E_{P_k^\infty}\left[2^n\one\{X_1=\cdots=X_{n-1}=0\}\one\{X_n=0\}\middle|\mathcal F_{n-1}\right]\\
&=2^n\one\{X_1=\cdots=X_{n-1}=0\}P_k(X_n=0)\\
&=2^n\one\{X_1=\cdots=X_{n-1}=0\}\cdot\tfrac12=E_{n-1}.
\end{align*}

\noindent Therefore $(E_n)$ is a martingale under each $P_k$ and is an $e$-process for $\Pcal$. Under $Q^\infty$, $E_n=2^n$ with probability one. Thus $(1/n)\log E_n=\log 2$. This is the optimal raw KL-infimum rate: for any $e$-process $M$ valid for $\Pcal$, the variational inequality gives $\E_Q\log M_n\le n\inf_k\KL(Q\|P_k)=n\log2$.
\end{proof}

\noindent Proposition~\ref{prop:weaklsc-not-necessary} illustrates why weak lower semicontinuity should be viewed as a sufficient weak-topological condition rather than a necessary condition.  Atoms and other discontinuous measurable features may yield power-one tests even when no weak neighborhood supplies a KL barrier.

\section{Aggregating tests to a consistent e-process}\label{sec:eproc}

The preceding stopping rules can be aggregated into one nonnegative adapted process that is valid under the null and diverges under every alternative.  The following lemma, in the spirit of \textcite{ruf2023composite}, converts a countable sequence of power-one tests into a divergent $e$-process.

\begin{lemma}\label{lem:count-tests-to-eprocess}
Suppose that there exists a sequence $(\alpha_k)_{k\ge 1}\subset(0,1)$ such that $\sum_{k=1}^{\infty}\alpha_k\le 1$.  For each $k\ge1$, suppose there exists a stopping rule $\tau_k$ satisfying $\sup_{P\in\Pcal}P^{\infty}(\tau_k<\infty)\le \alpha_k$ and $Q^{\infty}(\tau_k<\infty)=1$ for every $Q\in\Qcal$. Define $\sigma_k:=\tau_k\vee k$ and for $n\ge 0$,
\[
E_n:=\sum_{k=1}^{\infty}\one\{\sigma_k\le n\}.
\]
Then $(E_n)_{n\ge 0}$ is a nondecreasing $e$-process for $\Pcal$ and for every $Q\in\Qcal$,
\[
Q^{\infty}\left(\lim_{n\to\infty}E_n=\infty\right)=1.
\]
\end{lemma}

\begin{proof}
First, $(E_n)$ is well-defined and adapted. For each $n\ge0$, $\one\{\sigma_k\le n\}=0$ whenever $k>n$ because $\sigma_k\ge k$. It follows that,
\[
E_n=\sum_{k=1}^n \one\{\sigma_k\le n\}\le n.
\]
Thus $E_n$ is finite for every $n$.  Since $\sigma_k$ is a stopping rule it follows that $\{\sigma_k\le n\}\in \mathcal F_n$, so $E_n$ is $\mathcal F_n$ measurable. Because each indicator $n\mapsto \one\{\sigma_k\le n\}$ is nondecreasing, the process $(E_n)$ is nondecreasing.  We next show that the $e$-process property holds. To this end, take any $P\in\Pcal$ and any stopping rule $\tau$. Because $(E_n)$ is nondecreasing, it follows that $E_\tau=\lim_{m\to\infty}E_{\tau\wedge m}$. Also, for every $m\ge 1$ we know that,
\[
E_{\tau\wedge m}=\sum_{k=1}^{\infty}\one\{\sigma_k\le \tau\wedge m\},
\]
and all terms are necessarily nonnegative. By monotone convergence and Tonelli's theorem,
\begin{align*}
\E_{P^\infty}[E_\tau]&=\E_{P^\infty}\left[\lim_{m\to\infty}E_{\tau\wedge m}\right]\\
&=\lim_{m\to\infty}\E_{P^\infty}[E_{\tau\wedge m}]\\
&= \lim_{m\to\infty}\E_{P^\infty}\left[\sum_{k=1}^{\infty}\one\{\sigma_k\le \tau\wedge m\}\right]\\
&= \lim_{m\to\infty}\sum_{k=1}^{\infty}P^{\infty}(\sigma_k\le \tau\wedge m).
\end{align*}
For each $k$ and $m$, $\{\sigma_k\le \tau\wedge m\}\subseteq \{\sigma_k<\infty\}$. So, $P^{\infty}(\sigma_k\le \tau\wedge m)\le P^{\infty}(\sigma_k<\infty)$. Using this bound and the definition of $\sigma_k$ gives,
\[
\E_{P^\infty}[E_\tau]\le \sum_{k=1}^{\infty}P^{\infty}(\sigma_k<\infty) = \sum_{k=1}^{\infty} P^{\infty}(\tau_k<\infty)\le \sum_{k=1}^{\infty} \alpha_k \le 1.
\]
Since $P\in\Pcal$ and $\tau$ were arbitrary, $(E_n)$ is an $e$-process for $\Pcal$. It remains to prove divergence under each $Q\in\Qcal$. Fix such a $Q$.  By assumption, for every $k\ge 1$, $Q^{\infty}(\tau_k<\infty)=1$. Since $\sigma_k=\tau_k\vee k$, for every $k\ge 1$, $Q^{\infty}(\sigma_k<\infty)=1$ also. Taking a countable intersection gives an event $A$ with $Q^{\infty}(A)=1$ such that for every $\omega\in A$ and every $k\ge1$, $\sigma_k(\omega)<\infty$. For $\omega\in A$ and $m\ge1$, let $M_m(\omega):=\max_{1\le k\le m}\sigma_k(\omega)<\infty$. Whenever $n\ge M_m(\omega)$,
\[
E_n(\omega)=\sum_{k=1}^\infty \one\{\sigma_k(\omega)\le n\} \ge \sum_{k=1}^m \one\{\sigma_k(\omega)\le n\}=m.
\]
Since this holds for every $m\ge1$, $\lim_{n\to\infty}E_n(\omega)=\infty$. Therefore, 
\[
Q^{\infty}\left(\lim_{n\to\infty}E_n=\infty\right)=1.
\]
This proves the claim.
\end{proof}

\begin{corollary}\label{cor:main-implies-eproc}
Suppose that the assumptions of Theorem~\ref{thm:mainy} hold.  For each $k\ge1$, let $\alpha_k:=2^{-k}$. Choose a stopping rule $\tau_k$ given by Theorem~\ref{thm:mainy} with level at most $\alpha_k$ and power one against all $Q\in\Qcal$. Then the process $\sigma_k:=\tau_k\vee k$ and,
\[
E_n:=\sum_{k=1}^{\infty}\one\{\sigma_k\le n\},
\]
is a nondecreasing $e$-process for $\Pcal$ such that for every $Q\in\Qcal$,
\[
Q^{\infty}\left(\lim_{n\to\infty} E_n=\infty\right)=1.
\]
\end{corollary}

\begin{proof}
We know that $\sum_{k=1}^{\infty} 2^{-k}=1$. The claim follows from Lemma~\ref{lem:count-tests-to-eprocess} applied with $\alpha_k=2^{-k}$.    
\end{proof}

\noindent Theorem~\ref{thm:mainy} does not require uniform KL separation between the null and alternative classes.  It assumes only the pointwise condition $\Phi_\Pcal(Q)>0$ for each $Q\in\Qcal$; one may still have $\inf_{Q\in\Qcal}\Phi_\Pcal(Q)=0$.  Thus alternatives are allowed to approach the null class.

\section{REGROW \texorpdfstring{$e$}{e}-process against \texorpdfstring{$\Pcal^c$}{P-complement}}\label{sec:regrow}

% \section{REGROW e-processes}
% \label{sec:regrow}

\subsection{Why growth rates of e-processes?}

Classical sequential analysis is often formulated in terms of stopping times: one seeks a
stopping rule whose type-I error is controlled and whose stopping time is small, or finite
almost surely, under alternatives.  In the simple-vs-simple case, Wald's sequential
probability ratio test already reveals a deeper structure behind such procedures: the
likelihood ratio is not merely a statistic used at a stopping time, but a nonnegative
process whose value can be monitored continuously and thresholded whenever desired.  Later work of Robbins, Siegmund, Lai, and collaborators
developed power-one sequential tests and studied their stopping-time behavior, often
with the Kullback--Leibler information appearing as the leading exponential or
large-deviation quantity \parencite{darlingrobbins1968,robbins_siegmund_1974,lai1977,farrell1964}.

The modern language of e-values and e-processes isolates exactly the property of the
likelihood-ratio process that is needed for anytime-valid testing.  An e-process for a null
class \(\Pcal\) is a nonnegative adapted process \(E=(E_n)_{n\ge0}\) satisfying
\[
  \sup_{P\in\Pcal} \E_{P^\infty} E_\tau \le 1
\]
for every stopping time \(\tau\).  Consequently, Ville's inequality implies that
\[
  \tau_\alpha := \inf\{n:E_n\ge 1/\alpha\}
\]
is a level-\(\alpha\) sequential test, simultaneously for every \(\alpha\in(0,1)\).  Thus an
e-process is stronger than a single level-\(\alpha\) stopping rule: it is a reusable
evidence process that can be monitored continuously, thresholded at any desired level,
and remains valid under optional stopping and optional continuation.  This is the central
object in safe anytime-valid inference and related game-theoretic approaches to
sequential testing \parencite{shafer2011,ramdas2023,grunwald2024,ruf2023composite}.

For the preceding sections of this paper, divergence of an e-process under an alternative
is enough: if \(E_n\to\infty\) almost surely under \(Q\), then thresholding \(E\) yields
power one against \(Q\).  However, the \emph{power one} property does not distinguish between slow and fast
accumulation of evidence.  Once an e-process is viewed as the wealth process of a bettor
against the null, its long-run log-growth rate
\[
  \frac1n \log E_n
\]
becomes the natural analogue of efficiency.  In the simple-vs-simple case, the likelihood
ratio grows at rate \(\KL(Q\|P)\) under \(Q\), by the law of large numbers.  For composite
nulls and alternatives, this observation leads to growth-rate optimality criteria for
e-values and e-processes, including GRO, GROW, and REGROW
\parencite{grunwald2024,larsson2025numeraire,ramramdas2026optimalbettingwealth}.
The point of REGROW is not to maximize a single worst-case rate over a large alternative
class---which may be zero when alternatives approach the null---but instead to compare
the achieved growth at each fixed \(Q\) to the best rate attainable at that same \(Q\).

There is one important subtlety.  For a composite null \(\Pcal\), the correct pointwise
asymptotic betting-wealth benchmark is not always the raw quantity
\[
  \Phi_\Pcal(Q):=\inf_{P\in\Pcal}\KL(Q\|P).
\]
Recent work of Ram and Ramdas shows that the optimal fixed-\(Q\) asymptotic growth rate
is instead a bipolar inf-KL rate, defined using the effective \(n\)-sample null generated
by all valid nonnegative \(n\)-sample tests \parencite{ramramdas2026optimalbettingwealth}.
This bipolar inf-KL rate can be strictly smaller than \(\Phi_\Pcal(Q)\).  However, the two
rates coincide whenever \(Q\mapsto \Phi_\Pcal(Q)\) is weakly lower semicontinuous at \(Q\);
in particular, weak compactness of \(\Pcal\) is one useful sufficient condition.  Therefore,
under the local lower-semicontinuity assumptions used in this paper, the simpler
quantity \(\Phi_\Pcal(Q)\) remains the correct REGROW target.  The goal of this section is
to construct a single e-process that attains this pointwise optimal asymptotic growth
rate simultaneously over the alternative class.

An important point is that the optimality criterion in \cite{ramramdas2026optimalbettingwealth} is not the same as the criterion we use here, namely a $\liminf$ at all integer times. That is, \cite{ramramdas2026optimalbettingwealth} require non-defective stopping rules as a convention for the $e$-processes and the optimality criteria is phrased via a $\limsup$ along an associated block schedule. In addition to the fact that our $\REGROW$ criterion uses a $\liminf$ along every integer time, we also require validity at $\tau=\infty$ as well. We give a justification for this as follows. Denote the $n$-step $e$-power as $a_n(Q;\Pcal)=\sup_{Z\ge 0} \E_{Q^n}[\log Z]$ where $Z$ is an $e$-variable. Thanks to \cite{ramramdas2026optimalbettingwealth}, we know that products of optimal $e$-variables on independent blocks show $a_n$ is superadditive, and thus
\[
d(Q,\Pcal) :=\lim_{n\to\infty} \frac{a_n(Q)}{n}=\sup_{n\ge 1}\frac{a_n(Q,\Pcal)}{n}\in[0,\infty].
\]
Now, define
\[
\mathcal E_\infty(\Pcal) := \Bigl\{ E:\, \sup_{P\in\Pcal}\E_{P^\infty}[E_\tau]\le1 \text{ for every } \tau\in\mathbb N_0\cup\{\infty\} \Bigr\}.
\]
By convention, we will take $E_{\infty}=\limsup_{n\to\infty}E_n$.

\begin{proposition}\label{prop:integer-benchmark}
For every $Q\in\mathcal M_1(\X)$, we have that $\sup_{E\in \mathcal E_{\infty}(\Pcal)}\liminf_{n\to\infty} \frac1n \E_{Q^{\infty}}[\log E_n]=\dinfKL(Q;\Pcal)$. Specifically, every $E\in\mathcal E_{\infty}(\Pcal)$ satisfies
\[
\limsup_{n\to\infty} \frac1n \E_{Q^{\infty}}[\log E_n]\le \dinfKL(Q;\Pcal).
\]
And, there exists a non-decreasing $W^{(Q)}\in\mathcal E_{\infty}(\Pcal)$ such that
\[
\liminf_{n\to\infty}\frac1n \E_{Q^{\infty}}[\log W_n^{(Q)}]\ge \dinfKL(Q;\Pcal).
\]
\end{proposition}
\begin{proof}
Let us aptly choose strictly positive approximate (and finite) optimizers $Z_n$ such that $\E_{Q^n}[\log Z_n]\ge a_n(Q,\Pcal)-\varepsilon_n$. In addition, let us choose $c>0$  and $b_n>0$ such that $c+\sum_{n\ge 1}b_n\le 1$ and $\log(b_n)/n\to 0$. For example, this can be done by taking $c=0.5$ and $b_n=3/(\pi^2n^2)$. Now, let us set
\[
W_N^{(Q)} = c + \sum_{n=1}^N b_n Z_n(X_1,\dots, X_N).
\]
Then it follows that for every $P\in\Pcal$ and each stopping rule $\tau$ including $\infty$, we have
\begin{align*}
\E_{P^{\infty}} [W_\tau^{(Q)}]&=c+\sum_{n\ge 1} b_n\E_{P^{\infty}}[Z_n\mathbf{1}\{\tau\ge n\}]\\
&\le c+\sum_{n\ge 1}b_n \E_{P^{\otimes n}}[Z_n]\\
&\le c+\sum_{n=1}^{\infty} b_n = 1.
\end{align*}
Thus $W\in\mathcal E(\Pcal)$. Further since $W_N\ge b_N Z_N$, it follows that
\[
\frac{1}{N}\E_Q[\log W_n]\ge \frac{a_N(Q;\Pcal)}{N}-\frac{\varepsilon}{N}+\frac{\log b_N}{N}.
\]
Hence,
\[
\liminf_{N\to\infty}\frac{1}{N}\E_Q[\log W_N]\ge d(Q;\Pcal).
\]
Now, conversely, consider any $e$-process $E$. We know that its $N$th value will be an $N$-sample $e$-variable, hence $\E_Q[\log E_N]\le a_N(Q;\Pcal)$. As a consequence,
\[
\limsup_{N\to\infty} \frac{1}{N}\E_Q[\log E_n]\le d(Q;\Pcal).
\]
This shows the statement for all integer times. That is,
\[
\sup_{E\in\mathcal E(\Pcal)}\liminf_{N\to\infty}\frac1N \E_Q[\log E_N] = d(Q;\Pcal). 
\]
Therefore, this is the benchmark for our class of stopping rules, aligned with the previous work of \cite{ramramdas2026optimalbettingwealth}.
\end{proof}
% The power-one results above only require the pointwise separation quantity
% \[
%   \Phi(Q)=\inf_{P\in\mathcal P}\KL(Q\|P).
% \]
% For asymptotic growth-rate optimality of an e-process, however, the correct benchmark is slightly more subtle.
% Recent work of Ram and Ramdas shows that the optimal asymptotic betting-wealth growth rate is
% given by a bipolar infKL quantity, which can be strictly smaller than the raw KLinf value
% \(\Phi(Q)\) for nonclosed or otherwise irregular null classes. Under weak lower semicontinuity of
% \(\Phi\) at \(Q\), the two quantities coincide. This section uses the bipolar benchmark to define
% REGROW, and then shows that the weak-lsc construction attains it.

\subsection{The bipolar infKL benchmark}

Let \(S\) be a nonempty set of probability measures on a measurable space \((\Omega,\mathcal F)\).
Define its polar by
\[
  S^\circ
  :=
  \left\{
    E:\Omega\to[0,\infty]:
    \sup_{P\in S}\mathbb E_P E\le 1
  \right\}.
\]
The bipolar is
\[
  S^{\circ\circ}
  :=
  \left\{
    R\in\mathcal M_1(\Omega):
    \mathbb E_R E\le 1
    \text{ for every }E\in S^\circ
  \right\}.
\]
Plainly \(S\subseteq S^{\circ\circ}\). In this paper, the bipolar is used only to state the
general asymptotic growth-rate benchmark for e-processes.

For \(n\ge1\), write
\[
  \mathcal P^{(n)}
  :=
  \{P^n:P\in\mathcal P\}
  \subseteq\mathcal M_1(\mathcal X^n).
\]
For \(Q\in\mathcal M_1(\mathcal X)\), define
\[
  a_n(Q;\mathcal P)
  :=
  \inf_{R\in(\mathcal P^{(n)})^{\circ\circ}}
  \KL(Q^n\|R).
\]
The bipolar infKL rate is
\[
  d_{\mathrm{infKL}}(Q;\mathcal P)
  :=
  \lim_{n\to\infty}\frac1n a_n(Q;\mathcal P),
\]
whenever the limit exists. This is the optimal
asymptotic expected log-growth rate for e-processes testing \(\mathcal P\) against the fixed law
\(Q\).

In general,
\[
  d_{\mathrm{infKL}}(Q;\mathcal P)
  \le
  \Phi(Q).
\]
The inequality may be strict. However, if \(Q\mapsto\Phi(Q)\) is weakly lower semicontinuous at
\(Q\), then
\[
  d_{\mathrm{infKL}}(Q;\mathcal P)=\Phi(Q).
\]
Thus, on classes of alternatives where \(\Phi\) is locally weakly lower semicontinuous, the simpler
raw KLinf quantity remains the correct asymptotic REGROW target.

\subsection{Relative growth rate optimality}

Worst-case growth-rate optimality over the full complement $\Pcal^c$ is usually impossible, because alternatives may approach the null and the worst-case separation can be zero.  The appropriate goal is pointwise relative growth-rate optimality.  The pointwise target, however, must be the bipolar inf-KL rate $\dinfKL(Q;\Pcal)$ from Section~\ref{sec:warm}, not the raw KL-infimum in general.  Ram and Ramdas~\parencite{ramramdas2026optimalbettingwealth} prove that $\dinfKL(Q;\Pcal)$ is the optimal asymptotic betting-wealth growth rate for a fixed $Q$, and that $\dinfKL(Q;\Pcal)=\Phi_\Pcal(Q)$ whenever $\Phi_\Pcal(Q)<\infty$ is weakly lower semicontinuous at $Q$.

Let $\mathcal{E}(\Pcal)$ denote the class of all nonnegative adapted processes $E=(E_n)_{n\ge0}$ such that, for every stopping time $\tau$ taking values in $\mathbb N\cup\{\infty\}$,
\[
  \sup_{P\in\Pcal}\E_{P^\infty}[E_\tau]\le1,
\]
where $E_\infty:=\limsup_{n\to\infty}E_n$.  We use the convention $\log0=-\infty$.

For $Q\in\mathcal M_1(\X)$, define the corrected pointwise asymptotic GRO value by
\[
  \GRO_\infty(Q):=\dinfKL(Q;\Pcal).
\]
For $\Qcal\subseteq\mathcal M_1(\X)$, define the corrected asymptotic REGROW value by
\[
  \REGROW_\infty(\Qcal):=
  \sup_{E\in\mathcal{E}(\Pcal)}
  \inf_{Q\in\Qcal}
  \liminf_{n\to\infty}
  \left(
       \frac1n\E_{Q^\infty}[\log E_n]-\dinfKL(Q;\Pcal)
  \right).
\]
An $e$-process is pointwise pathwise REGROW on $\Qcal$ if
\[
  \liminf_{n\to\infty}\frac1n\log E_n
  \ge \dinfKL(Q;\Pcal)
  \qquad Q^\infty\text{-a.s. for every }Q\in\Qcal.
\]

Now, the actual $\REGROW$ value will only be well-defined on classes where the $\GRO$ benchmark value is finite. We now present our $\REGROW$ theorem as follows.

\begin{theorem}\label{regrow-localweaklsc}
Let $\X$ be Polish. Let $\varnothing\ne \Pcal\subseteq\mathcal M_1(\X)$ and denote $\Phi_{\Pcal}(R)=\inf_{P\in\Pcal}\KL(R\mid P)$. Let $\varnothing\ne \Qcal\subseteq \mathcal M_1(\X)$ and suppose that $\Phi_{\Pcal}$ is weakly lower semicontinuous at every $Q\in\Qcal$. Then, there exists a finite-valued, nondecreasing process $E^{\star}\in\mathcal E_{\infty}$ with $E_n^{\star}\ge 3/4$ for every $n$ such that for all $Q\in\Qcal$, almost surely,
\[
\liminf_{n\to\infty}\frac1n\log E_n^\star\ge \Phi_{\Pcal}(Q)
\]
And,
\[
\lim_{n\to\infty}\frac1n \E_{Q^{\infty}}[\log E_n^{\star}]=\Phi_{\Pcal}(Q)=\dinfKL(Q;\Pcal).
\]
Consequently, if $\Qcal_{\mathrm{fin}}=\{Q\in\Qcal:\Phi_{\Pcal}(Q)<\infty\}$ is nonempty, then it follows that $\REGROW_{\infty}(\Qcal_{\mathrm{fin}}; \Pcal)=0$ and $E^{\star}$ attains this value. Now if in addition we have that $\Phi_{\Pcal}>0$ for every $Q\in\Qcal$, then it follows that $E_n^{\star}\rightarrow \infty$ almost surely under $Q^{\infty}$. 
\end{theorem}

\begin{proof}
Fix a bounded-Lipschitz metric \(d_{\rm BL}\) that metrizes weak convergence on
\(\mathcal M_1(\mathcal X)\). Let \((r_k)_{k\ge1}\) be an enumeration of
\(\mathbb Q_{>0}\). For each \(k\), define
\[
  \mathcal Q_k:=\{Q\in\mathcal Q:\Phi(Q)>r_k\}.
\]
For every \(Q\in\mathcal Q_k\), weak lower semicontinuity of \(\Phi\) at \(Q\) implies that there
exists a weakly open neighborhood \(V_{k,Q}\ni Q\) such that
\[
  \Phi(R)>r_k
  \qquad\text{for all }R\in V_{k,Q}.
\]
Choose \(\delta_{k,Q}>0\) sufficiently small that the closed \(d_{\rm BL}\)-ball
\[
  C_{k,Q}:=\{R\in\mathcal M_1(\mathcal X):d_{\rm BL}(R,Q)\le \delta_{k,Q}\}
\]
is contained in \(V_{k,Q}\), and write
\[
  U_{k,Q}:=\{R\in\mathcal M_1(\mathcal X):d_{\rm BL}(R,Q)<\delta_{k,Q}\}.
\]
Then \(U_{k,Q}\) is a weakly open neighborhood of \(Q\), \(C_{k,Q}\) is weakly closed and convex,
and
\[
  \inf_{R\in C_{k,Q}}\Phi(R)\ge r_k.
\]
The family \(\{U_{k,Q}:Q\in\mathcal Q_k\}\) is an open cover of \(\mathcal Q_k\). Since
\(\mathcal X\) is Polish, \(\mathcal M_1(\mathcal X)\) with the weak topology is separable and
metrizable, hence Lindelöf. Therefore this cover has a countable subcover. We denote it by
\[
  (U_{k,j})_{j\in J_k},
\]
where \(J_k\subseteq\mathbb N\) is finite or countably infinite, and let \(C_{k,j}\) be the corresponding
closed \(d_{\rm BL}\)-ball. Thus, for all \(k\) and \(j\in J_k\),
\[
  C_{k,j}\text{ is weakly closed and convex,}
  \qquad
  \inf_{R\in C_{k,j}}\Phi(R)\ge r_k.
\]

Let
\(
  v_t:=\frac{6}{\pi^2t^2},
   t\ge1,
\)
so that \(\sum_{t\ge1}v_t=1\). Define
\[
  S:=\sum_{k\ge1}2^{-k}e^{-r_k^2}\in(0,\infty),
\]
and, for \(j\in J_k\),
\[
  w_{k,j}:=\frac{2^{-k-j}e^{-r_k^2}}{4S}.
\]
Then
\[
  \sum_{k\ge1}\sum_{j\in J_k} w_{k,j}\le \frac14.
\]
For \(n\ge0\), define
\[
  E^\star_n
  :=
  \frac34
  +
  \sum_{k\ge1}\sum_{j\in J_k}
  w_{k,j}\sum_{t=1}^n
  v_t e^{r_kt}\mathbf 1\{\widehat Q_t\in C_{k,j}\},
\]
where \(\widehat Q_t\) is the empirical measure of \(X_1,\ldots,X_t\).

First we check that \(E^\star_n<\infty\) for every deterministic \(n\). Indeed,
\[
\begin{aligned}
  \sum_{k\ge1}\sum_{j\in J_k}
  w_{k,j}\sum_{t=1}^n v_t e^{r_kt}
  &\le
  \sum_{k\ge1}\sum_{j\in J_k}
  w_{k,j} e^{r_kn}  \\
  &\le
  \frac1{4S}
  \sum_{k\ge1}2^{-k}e^{-r_k^2+r_kn}
  \sum_{j\ge1}2^{-j}.
\end{aligned}
\]
Since
\(
  -r_k^2+r_kn\le \frac{n^2}{4},
\)
the final display is at most \(e^{n^2/4}/(4S)<\infty\).

We next verify the e-process property. Fix \(P\in\mathcal P\) and an arbitrary stopping time
\(\tau\), with the convention \(E^\star_\infty:=\lim_{n\to\infty}E^\star_n\), which exists because
\(E^\star_n\) is nondecreasing in \(n\). By Tonelli's theorem,
\[
\begin{aligned}
  \mathbb E_{P^\infty}E^\star_\tau
  &=
  \frac34
  +
  \sum_{k\ge1}\sum_{j\in J_k}
  w_{k,j}
  \sum_{t\ge1}
  v_t e^{r_kt}
  P^\infty(\tau\ge t,\widehat Q_t\in C_{k,j})  \\
  &\le
  \frac34
  +
  \sum_{k\ge1}\sum_{j\in J_k}
  w_{k,j}
  \sum_{t\ge1}
  v_t e^{r_kt}
  P^t(\widehat Q_t\in C_{k,j}).
\end{aligned}
\]
Since \(C_{k,j}\) is weakly closed and convex, Csiszár's nonasymptotic Sanov bound gives
\[
  P^t(\widehat Q_t\in C_{k,j})
  \le
  \exp\left\{
    -t\inf_{R\in C_{k,j}}\KL(R\|P)
  \right\}.
\]
For every \(R\in C_{k,j}\),
\[
  \KL(R\|P)\ge \inf_{P'\in\mathcal P}\KL(R\|P')=\Phi(R),
\]
and hence
\[
  \inf_{R\in C_{k,j}}\KL(R\|P)
  \ge
  \inf_{R\in C_{k,j}}\Phi(R)
  \ge r_k.
\]
Therefore
\[
  P^t(\widehat Q_t\in C_{k,j})\le e^{-r_kt}.
\]
Substituting this bound yields
\[
\begin{aligned}
  \mathbb E_{P^\infty}E^\star_\tau
  &\le
  \frac34
  +
  \sum_{k\ge1}\sum_{j\in J_k}
  w_{k,j}
  \sum_{t\ge1}v_t e^{r_kt}e^{-r_kt}  \\
  &=
  \frac34
  +
  \sum_{k\ge1}\sum_{j\in J_k}w_{k,j}
  \le 1.
\end{aligned}
\]
Since \(P\in\mathcal P\) and \(\tau\) were arbitrary, \(E^\star\) is an e-process for
\(\mathcal P\).

We now prove the pathwise lower bound. Fix \(Q\in\mathcal Q\), and let \(r_k<\Phi(Q)\). Then
\(Q\in\mathcal Q_k\), so there is some \(j\in J_k\) such that \(Q\in U_{k,j}\). By the strong law
for empirical measures on Polish spaces, \(\widehat Q_n\Rightarrow Q\) \(Q^\infty\)-almost surely.
Since \(U_{k,j}\) is a weak neighborhood of \(Q\), it follows that, on a \(Q^\infty\)-almost sure event,
\[
  \widehat Q_n\in U_{k,j}\subseteq C_{k,j}
  \qquad\text{eventually}.
\]
Hence, eventually,
\[
  E^\star_n
  \ge
  w_{k,j}v_n e^{r_kn}.
\]
Therefore
\[
  \liminf_{n\to\infty}\frac1n\log E^\star_n
  \ge
  r_k
  +
  \liminf_{n\to\infty}\frac1n\log w_{k,j}
  +
  \liminf_{n\to\infty}\frac1n\log v_n
  =
  r_k,
\]
because \(w_{k,j}>0\) is fixed and \(n^{-1}\log v_n\to0\). Since this holds for every rational
\(r_k<\Phi(Q)\), we obtain
\[
  \liminf_{n\to\infty}\frac1n\log E^\star_n
  \ge
  \Phi(Q)
  \qquad Q^\infty\text{-a.s.}
\]
If \(\Phi(Q)=+\infty\), the same argument applied to every \(r_k\in\mathbb Q_{>0}\) gives an infinite
lower bound.

It remains to prove the expected-log optimality statement for \(Q\in\mathcal Q_{\rm fin}\). Let
\(E=(E_n)_{n\ge0}\) be any e-process for \(\mathcal P\). For every deterministic \(n\), \(E_n\) is an
\(n\)-sample e-variable, so
\[
  \mathbb E_{P^n}E_n\le1
  \qquad\text{for every }P\in\mathcal P.
\]
By the variational inequality for relative entropy, for every \(P\in\mathcal P\),
\[
  \mathbb E_{Q^n}\log E_n
  \le
  \KL(Q^n\|P^n)+\log\mathbb E_{P^n}E_n
  \le
  n\KL(Q\|P).
\]
Taking the infimum over \(P\in\mathcal P\) gives
\[
  \frac1n\mathbb E_{Q^\infty}\log E_n
  \le
  \Phi(Q).
\]
Thus no e-process can have asymptotic expected log-growth exceeding \(\Phi(Q)\) at such a \(Q\).

Applying the pathwise lower bound to \(E^\star\), and using \(E^\star_n\ge3/4\), Fatou's lemma gives
\[
\begin{aligned}
  \Phi(Q)
  &\le
  \mathbb E_{Q^\infty}
  \left[
    \liminf_{n\to\infty}
    \frac1n\log E^\star_n
  \right]  \\
  &\le
  \liminf_{n\to\infty}
  \frac1n\mathbb E_{Q^\infty}\log E^\star_n
  \le
  \limsup_{n\to\infty}
  \frac1n\mathbb E_{Q^\infty}\log E^\star_n
  \le
  \Phi(Q).
\end{aligned}
\]
Therefore
\[
  \lim_{n\to\infty}
  \frac1n\mathbb E_{Q^\infty}\log E^\star_n
  =
  \Phi(Q),
\]
as claimed.
\end{proof}

\begin{remark}[Composite REGROW beyond weak compactness]
For a fixed $Q$, the bipolar inf-KL theory gives an asymptotically optimal betting strategy.  The construction above shows that, under weak compactness, the weak-lower-semicontinuity and Sanov-neighborhood argument supplies a single process that is pointwise optimal over the full complement $\Pcal^c$.
\end{remark}

\section{Nonparametric examples where \texorpdfstring{$\Phi$}{Phi} is weakly lower semicontinuous}\label{sec:examples}

We discuss three classes of examples here, with many more left to Appendix~\ref{app:more-examples}.

\subsection{Weakly compact \texorpdfstring{$\Pcal$}{P} and KL-local compactness}

\begin{lemma}\label{lem:phi-lsc}
If $\Pcal\subseteq\mathcal{M}_1(\X)$ is weakly compact, then $Q \mapsto \Phi(Q)$ is weakly lower semicontinuous.
\end{lemma}

\begin{proof}
Let $R_k\Rightarrow R$ in $\mathcal{M}_1(\X)$ and set $L:=\liminf_{k\to\infty}\Phi(R_k)\in[0,\infty]$.  Choose a subsequence $(k_j)_{j\ge1}$ such that $\Phi(R_{k_j})\to L$.  For each $j$, choose an almost-minimizer $P_{k_j}\in\Pcal$ satisfying $\KL(R_{k_j}\|P_{k_j})\le \Phi(R_{k_j})+j^{-1}$.  By weak compactness of $\Pcal$, the sequence $(P_{k_j})$ has a weakly convergent subsequence, still denoted $(P_{k_j})$, with limit $P_\infty\in\Pcal$.  Since $R_{k_j}\Rightarrow R$, Lemma~\ref{lem:kl-lsc} gives,
\(
\KL(R\|P_\infty)\le \liminf_{j\to\infty}\KL(R_{k_j}\|P_{k_j}).
\)
Let us now use the almost minimality and the fact that  $\Phi(R)\le \KL(R\|P_\infty)$. Doing this gives us that
\[
\Phi(R)\le \liminf_{j\to\infty}\KL(R_{k_j}\|P_{k_j})\le \liminf_{j\to\infty}\left(\Phi(R_{k_j})+\frac{1}{j}\right)=\liminf_{k\to\infty}\Phi(R_k),
\]
which is exactly lower semicontinuity of $\Phi$.
\end{proof}

As a prototypical example, consider the bounded mean problem: let $\Omega =[0,1]$ and let $\Pcal$ be the set of distributions with mean at most $a$ for some $a \in (0,1)$. This is a nonparametric class of distributions that is easily checked to be weakly compact. We give a more nontrivial example below.

\paragraph{Wasserstein Balls.}
Let \((\X,d)\) be a Polish metric space.  For \(p\in[1,\infty)\), define the \(p\)-Wasserstein distance, with value \(+\infty\) allowed, by
\[
  \Wass_p(P,R)^p
  :=\inf_{\pi\in\Gamma(P,R)}\int_{\X\times\X}d(x,y)^p\,\dd\pi(x,y),
\]
where \(\Gamma(P,R)\) is the set of couplings of \(P\) and \(R\).  For \(p=1\), Kantorovich--Rubinstein duality identifies \(\Wass_1\) as the IPM generated by the class of 1-Lipschitz functions, modulo irrelevant additive constants.  
Define
\[
  \Pcal_{\Wass_p}(\rho)
  :=\{P\in\M(\X):\Wass_p(P,P_0)\le\rho\}.
\]
For \(p>1\), \(\Wass_p\) is not an IPM, but the same lower-semicontinuity conclusion below applies.

\begin{proposition}[Wasserstein balls on proper spaces]\label{prop:wasserstein-lsc}
Assume that \((\X,d)\) is proper, meaning that closed metric balls are compact.  Let \(P_0\in\M(\X)\) have finite \(p\)-th moment, and fix \(\rho<\infty\).  
Then \(\Pcal_{\Wass_p}(\rho)\) is weakly compact.  Consequently,
\(
  Q\mapsto
  \Phi_{\Pcal_{\Wass_p}(\rho)}(Q)
\)
is weakly lower semicontinuous on \(\M(\X)\).  Moreover \(\Phi_{\Pcal_{\Wass_p}(\rho)}(Q)>0\) for every \(Q\notin\Pcal_{\Wass_p}(\rho)\).  
\end{proposition}

A weak compactness proof for Wasserstein balls under properness is also given by Yue, Kuhn, and Wiesemann \cite[Theorem 1]{yue2022linear}, but we provide a self-contained proof in the Appendix for completeness.

\paragraph{From weak compactness to KL-local compactness} The rest of this section contains examples where $\Pcal$ is not weakly compact. These are useful because the null may contain arbitrary nuisance distributions, escaping parameters, or mass that can move to infinity.  The common mechanism is not global compactness of the null, but compactness of the nearly KL-optimal denominators for a fixed weakly convergent alternative sequence.

\begin{definition}[KL-local compactness relative to weak convergence]\label{def:kl-local-compact}
A null class \(\Pcal\subseteq\M(\X)\) is called \emph{KL-locally compact} if the following holds: whenever \(Q_n\Rightarrow Q\) in \(\M(\X)\), \(P_n\in\Pcal\), and
\(
  \sup_n \KL(Q_n\|P_n)<\infty,
\)
there is a subsequence \(P_{n_k}\Rightarrow P\) with \(P\in\Pcal\).
\end{definition}

\begin{proposition}[A noncompact sufficient condition]\label{prop:kl-local-compact}
If \(\Pcal\) is KL-locally compact, then \(Q\mapsto\Phi_{\Pcal}(Q)\) is weakly lower semicontinuous on \(\M(\X)\).
\end{proposition}
\begin{proof}
Let \(Q_n\Rightarrow Q\), and pass to a subsequence along which \(\Phi_{\Pcal}(Q_n)\) converges to
\[
  L:=\liminf_n\Phi_{\Pcal}(Q_n).
\]
If \(L=+\infty\), there is nothing to show.  Otherwise choose \(P_n\in\Pcal\) such that
\[
  \KL(Q_n\|P_n)\le \Phi_{\Pcal}(Q_n)+1/n.
\]
The sequence \((\KL(Q_n\|P_n))_n\) is bounded, so KL-local compactness gives a subsequence, still denoted \((P_n)\), with \(P_n\Rightarrow P\in\Pcal\).  Joint weak lower semicontinuity of relative entropy gives
\(
  \Phi_{\Pcal}(Q)
  \le \KL(Q\|P)
  \le \liminf_n \KL(Q_n\|P_n)
  = L.
\)
This is exactly weak lower semicontinuity.
\end{proof}

If \(\Pcal\) is weakly compact and weakly closed, then it is KL-locally compact.  The converse fails.  KL-local compactness only controls denominator sequences that remain at finite KL distance from a tight, weakly convergent alternative sequence.  The null may still contain many other sequences that escape to infinity.

\subsection{\texorpdfstring{\(f\)}{f}-divergence balls}

Let \(f:[0,\infty)\to[0,\infty]\) be a proper, convex, lower semicontinuous function with \(f(1)=0\).  Define
\(
f'_\infty:=\lim_{t\to\infty}\frac{f(t)}{t}\in[0,\infty].
\)
For probability measures \(P\) and \(P_0\), let \(P=hP_0+P^\perp\) be the Lebesgue decomposition with \(P^\perp\perp P_0\).  The Csiszar divergence is
\begin{equation}\label{eq:fdiv}
  D_f(P\|P_0):=\int f(h)d P_0+f'_\infty P^\perp(\X).
\end{equation}
For \(r\ge0\), write
\begin{equation}\label{eq:fball}
  \Pcal:=\{P\in\mathcal M(\X):D_f(P\|P_0)\le r\}.
\end{equation}
Well known examples include the total variation, squared Hellinger, forward and reverse KL, $\chi^2$, and power divergence balls. Reverse-orientation balls are included because \(D_f(P_0\|P)=D_{\widetilde f}(P\|P_0)\), where \(\widetilde f(t):=t f(1/t)\), with the usual endpoint conventions.  Finite intersections of constraints of the form \(D_{f_j}(P\|P_j)\le r_j\) are also covered by the same proof below.

\begin{lemma}[Weak lower semicontinuity of \(f\)-divergence]\label{lem:f-lsc}
For every closed convex generator \(f\) as above, the map \(P\mapsto D_f(P\|P_0)\) is weakly lower semicontinuous.
\end{lemma}

\begin{lemma}[Strict positivity off an \(f\)-ball]\label{lem:fball-positive}
If \(Q\notin\Pcal\), then \(\Phi_{\Pcal}(Q)>0\).
\end{lemma}

\begin{theorem}[Weak lower semicontinuity for \(f\)-balls]\label{thm:fball-lsc}
Let \(\X\) be Polish, let \(P_0\in\M(\X)\), and let
\(f:[0,\infty)\to[0,\infty]\) satisfy the properties needed to define the Csiszar divergence \(D_f(P\|P_0)\) by \eqref{eq:fdiv}.  Let $\Pcal$ be the f-divergence ball in~\eqref{eq:fball}.
Then
\(
  Q\mapsto \Phi_{\Pcal}(Q)
\)
is weakly lower semicontinuous on \(\M(\X)\).
\end{theorem}

\begin{proof}
Let \(Q_n\Rightarrow Q\).  If \(\liminf_n\Phi_{\Pcal}(Q_n)=\infty\), there is nothing to prove.  Passing to a subsequence, assume that this liminf is finite and is attained along the whole subsequence.  Choose \(P_n\in\Pcal\) such that
\(
  \KL(Q_n\|P_n)\le \Phi_{\Pcal}(Q_n)+1/n.
\)
Embed \(\X\) as a Borel subset of a compact metric space \(K\), and view all measures as measures on \(K\).  Since \(K\) is compact, a further subsequence satisfies \(P_n\Rightarrow\bar P\in\M(K)\).  Also \(Q_n\Rightarrow Q\) in \(\M(K)\), because weak convergence on \(\X\) is preserved by the embedding.  By Lemmata \ref{lem:kl-lsc} and \ref{lem:f-lsc},
\[
  \KL(Q\|\bar P)
  \le \liminf_n\KL(Q_n\|P_n),
  \qquad
  D_f(\bar P\|P_0)
  \le \liminf_nD_f(P_n\|P_0)
  \le r.
\]
Apply Lemma~\ref{lem:recycle} to obtain \(P^\circ\in\M(\X)\) such that \(D_f(P^\circ\|P_0)\le r\) and \(\KL(Q\|P^\circ)\le\KL(Q\|\bar P)\).  Then
\[
  \Phi_{\Pcal}(Q)
  \le\KL(Q\|P^\circ)
  \le\KL(Q\|\bar P)
  \le\liminf_n\KL(Q_n\|P_n)
  =\liminf_n\Phi_{\Pcal}(Q_n).
\]
This proves weak lower semicontinuity.
\end{proof}

\subsection{Integral Probability Metric Balls}

Let \(\Fcal\) be a class of real-valued measurable functions on \(\X\).  Its integral probability metric is
\begin{equation}\label{eq:ipm}
  \gamma_{\Fcal}(P,R)
  :=\sup_{f\in\Fcal}\left|\int f\dd P-\int f\dd R\right|,
\end{equation}
and the corresponding null ball is
\begin{equation}\label{eq:ipm-ball}
  \Pcal_{\Fcal}(\rho)
  :=\{P\in\M(\X):\gamma_{\Fcal}(P,P_0)\le \rho\}.
\end{equation}
 For arbitrary measurable \(\Fcal\), IPM balls need not be weakly closed, and without some additional closedness, tightness, or other property, lower semicontinuity of \(\Phi\) should not be expected.

 If every \(f\in\Fcal\) is bounded and continuous, then \(\Pcal_{\Fcal}(\rho)\) is weakly closed, because it is the intersection over \(f\in\Fcal\) of the closed slabs
\(
  \left\{P:\left|\int f\dd P-\int f\dd P_0\right|\le\rho\right\}.
\)
Thus \(\Pcal_{\Fcal}(\rho)\) is weakly compact whenever it is also uniformly tight.  Uniform tightness is automatic if \(\X\) is compact.  More generally, it follows if the IPM constraint controls a coercive moment: for example, if \(V:\X\to[0,\infty]\) has compact sublevel sets and the ball implies \(\sup_{P\in\Pcal_{\Fcal}(\rho)}\int V\dd P<\infty\), then Prokhorov's theorem gives tightness.

\begin{lemma}[Strict positivity for bounded IPMs]\label{lem:ipm-positive}
Assume \(\sup_{f\in\Fcal}\|f\|_\infty<\infty\).  If \(Q\notin\Pcal_{\Fcal}(\rho)\), then \(\Phi_{\Pcal_{\Fcal}(\rho)}(Q)>0\).
\end{lemma}

\begin{proposition}[Precompact continuous IPM classes]\label{prop:precompact-F}
Assume \(\rho>0\).  Suppose that \(\Fcal\subset C_b(\X)\) is uniformly bounded and totally bounded in the sup norm.  Then \(Q\mapsto\Phi_{\Pcal_{\Fcal}(\rho)}(Q)\) is weakly lower semicontinuous and it is strictly positive on \(\Pcal_{\Fcal}(\rho)^c\).
\end{proposition}

On a compact sample space, many familiar classes are totally bounded by Arzela--Ascoli: uniformly bounded equicontinuous classes, bounded-Lipschitz balls, and unit balls of reproducing kernel Hilbert spaces (RKHSs) with continuous kernels.  On a noncompact sample space, total boundedness in the global sup norm is much stronger and often fails; the MMD and KS examples below use more tailored recycling arguments.
To begin with the MMD, let \(k:\X\times\X\to\RR\) be a bounded continuous positive-definite kernel with separable RKHS \(\mathcal H\).  Let \(\varphi(x):=k(x,\cdot)\in\mathcal H\), and assume \(\sup_x\|\varphi(x)\|_{\mathcal H}<\infty\).  The MMD is the IPM over the RKHS unit ball:
\[
  \MMD_k(P,R)
  :=\sup_{\|h\|_{\mathcal H}\le1}\left|\int h\dd P-\int h\dd R\right|
  =\left\|\mu_P-\mu_R\right\|_{\mathcal H},
  \qquad
  \mu_P:=\int\varphi\dd P.
\]

\begin{theorem}[MMD balls]\label{prop:mmd-lsc}
For every \(\rho>0\), the map
\(
Q\mapsto\Phi_{\Pcal}(Q)
\)
is weakly lower semicontinuous.  It is strictly positive whenever \(\MMD_k(Q,P_0)>\rho\).
\end{theorem}

No characteristic assumption is needed for lower semicontinuity.  Characteristicness determines whether \(\MMD_k(P,Q)=0\) implies \(P=Q\), but the REGROW theorem only needs lower semicontinuity of the distance-to-null benchmark and positivity outside the named MMD ball.

On \(\X=\RR\), the one-dimensional Kolmogorov--Smirnov distance from \(P\) to \(P_0\) is
\[
  d_{\KS}(P,P_0):=\sup_{x\in\RR}|F_P(x)-F_{P_0}(x)|,
  \qquad F_P(x):=P((-
\infty,x]).
\]
This is the IPM generated by lower-half-line indicators \(\one\{(-\infty,x]\}:x\in\RR\), which are bounded but not continuous.  Thus the previous continuous-function propositions do not apply directly.  Nevertheless, the same recycling idea works.

\begin{theorem}[Kolmogorov--Smirnov balls]\label{thm:ks-lsc}
For every \(\rho>0\), the map
\(
  Q\mapsto\Phi_\Pcal(Q)
\)
is weakly lower semicontinuous on \(\M(\RR)\).  It is strictly positive whenever \(d_{\KS}(Q,P_0)>\rho\).  For \(\rho=0\), the null is \(\{P_0\}\), so the same conclusion reduces to the usual weak lower semicontinuity of \(\KL(Q\|P_0)\).
\end{theorem}

\section{Conclusion}\label{sec:conc}
We have given a general sufficient condition for power-one sequential testing of an i.i.d. law on a Polish sample space against a composite null.  If $\Phi_\Pcal(Q)>0$ and $\Phi_\Pcal$ is weakly lower semicontinuous at each alternative $Q$, then a single level-$\alpha$ stopping rule has power one against the full alternative class.  Weak compactness of $\Pcal$ is the cleanest sufficient condition and yields testing against the entire complement $\Pcal^c$ without domination, convexity, parametric structure, or uniform separation.

The paper also separates two notions that are easily conflated.  Weak lower semicontinuity is a sufficient weak-topological route to power one and to REGROW optimality, but it is not necessary for consistency.  Discontinuous events can yield power-one tests even when weak neighborhoods fail.

For growth-rate optimality, the correct pointwise benchmark is the bipolar inf-KL rate $\dinfKL(Q;\Pcal)$, not the raw KL-infimum in general.  Under the weak-lower-semicontinuity assumptions emphasized in this paper, these two rates coincide, and the REGROW construction achieves the true optimal rate.  Outside these assumptions, future work should identify tractable conditions that allow one betting strategy to approximate the bipolar inf-KL rate simultaneously over a composite alternative.

The examples show that the lower-semicontinuity condition is broad: it includes weakly compact nonparametric nulls, divergence balls, IPM and Wasserstein balls under suitable hypotheses, and several noncompact semiparametric nulls with unrestricted nuisance components.  A useful direction for applications is to turn the existential Sanov-neighborhood constructions into computationally explicit betting strategies, for example by using predictable plug-in witnesses and convex dual calibrations for divergence, IPM, and Wasserstein constraints.

\printbibliography
\appendix

\section{Omitted proofs from Section~\ref{sec:examples}}

\subsection{Wasserstein ball}

\begin{proof}[Proof of Proposition~\ref{prop:wasserstein-lsc}]
Fix \(x_0\in\X\), and put \(m_0:=\int d(x,x_0)^p\,\dd P_0(x)<\infty\).  If \(P\in\Pcal_{\Wass_p}(\rho)\), then by the triangle inequality for \(\Wass_p\),
\[
  \left(\int d(x,x_0)^p\,\dd P(x)\right)^{1/p}
  =\Wass_p(P,\delta_{x_0})
  \le \Wass_p(P,P_0)+\Wass_p(P_0,\delta_{x_0})
  \le \rho+m_0^{1/p}.
\]
Hence the ball has uniformly bounded \(p\)-th moments.  Since \((\X,d)\) is proper, the closed metric ball \(\overline B(x_0,R)\) is compact, and Markov's inequality gives
\[
  \sup_{P\in\Pcal_{\Wass_p}(\rho)}P(\overline B(x_0,R)^c)
  \le \frac{(\rho+m_0^{1/p})^p}{R^p}
  \to0.
\]
Thus \(\Pcal_{\Wass_p}(\rho)\) is uniformly tight.
It remains to check weak closedness.  Let \(P_n\Rightarrow P\) and \(\Wass_p(P_n,P_0)\le\rho\).  Choose couplings \(\pi_n\in\Gamma(P_n,P_0)\) with
\[
  \int d(x,y)^p\,\dd\pi_n(x,y)\le \rho^p+1/n.
\]
The family \((\pi_n)\) is tight because its marginals are tight.  Passing to a subsequence, \(\pi_n\Rightarrow\pi\), and the marginal maps show that \(\pi\in\Gamma(P,P_0)\).  Since \((x,y)\mapsto d(x,y)^p\) is lower semicontinuous,
\[
  \Wass_p(P,P_0)^p
  \le \int d(x,y)^p\,\dd\pi(x,y)
  \le \liminf_n\int d(x,y)^p\,\dd\pi_n(x,y)
  \le \rho^p.
\]
Therefore the Wasserstein ball is weakly closed.  Uniform tightness plus weak closedness gives weak compactness by Prokhorov's theorem, implying lower semicontinuity of \(\Phi\).

For positivity, suppose \(Q\notin\Pcal_{\Wass_p}(\rho)\) but \(\Phi(Q)=0\).  Then there are \(P_m\in\Pcal_{\Wass_p}(\rho)\) with \(\KL(Q\|P_m)\to0\).  Pinsker's inequality gives \(d_{\TV}(Q,P_m)\to0\), hence \(P_m\Rightarrow Q\).  Since the ball is weakly closed, \(Q\in\Pcal_{\Wass_p}(\rho)\), a contradiction.
\end{proof}

\subsection{f-divergence balls}

\begin{proof}[Proof of Lemma~\ref{lem:f-lsc}]
The standard variational formula for Csiszar divergences gives
\[
  D_f(P\|P_0)
  =\sup_{g\in C_b(\X)}\left\{\int gd P-\int f^*(g)d P_0\right\},
\]
where \(f^*(u)=\sup_{t\ge0}\{ut-f(t)\}\) is the convex conjugate, and the supremum may be restricted to bounded continuous \(g\) for Polish spaces.  For fixed \(g\), \(\int gd P\) is weakly continuous and the second term is fixed.  Hence the supremum is weakly lower semicontinuous.
\end{proof}

\begin{proof}[Proof of Lemma~\ref{lem:fball-positive}]
Suppose, to the contrary, that \(\Phi_{\Pcal}(Q)=0\).  Then there are \(P_m\in\Pcal\) such that \(\KL(Q\|P_m)\to0\).  Pinsker's inequality gives \(d_{\mathrm{TV}}(Q,P_m)\to0\).  Total-variation convergence gives \(L^1\) convergence of densities with respect to a common dominating measure; passing to an almost-surely convergent subsequence and applying Fatou's lemma to the closed convex integrand in \eqref{eq:fdiv} yields
\[
  D_f(Q\|P_0)\le\liminf_m D_f(P_m\|P_0)\le r.
\]
Thus \(Q\in\Pcal\), a contradiction.
\end{proof}

\paragraph{Compactification convention for Polish spaces.}
The compactification used in the next lemma adds no assumption beyond Polishness of \(\X\).  Indeed, every Polish space admits a homeomorphic embedding \(\iota:\X\to[0,1]^{\mathbb N}\); if \(K\) is the closure of \(\iota(\X)\), then \(K\) is compact metric and \(\iota(\X)\) is Borel in \(K\).  After identifying \(\X\) with its image in \(K\), weak convergence on \(\X\) implies weak convergence of the corresponding pushforward measures on \(K\), because every bounded continuous function on \(K\) restricts to a bounded continuous function on \(\X\).  Thus, whenever this section assumes that \(\X\) is Polish, we may choose such a compact metric \(K\) without adding a separate topological hypothesis.

\begin{lemma}[Recycling escaped mass]\label{lem:recycle}
Let \(\X\) be embedded as a Borel subset of a compact metric space \(K\) as in the preceding compactification convention, and extend \(P_0\) to \(K\) by assigning zero mass to \(K\setminus\X\).  If \(\bar P\in\M(K)\), then there exists \(P^\circ\in\M(\X)\) such that
\[
  D_f(P^\circ\|P_0)\le D_f(\bar P\|P_0)
  \quad\text{and}\quad
  \KL(Q\|P^\circ)\le\KL(Q\|\bar P)
\]
for every \(Q\in\M(\X)\), where \(Q\) is also viewed as a measure on \(K\).
\end{lemma}

\begin{proof}
Let \(m:=\bar P(K\setminus\X)\).  Fix any \(x_0\in\X\), and define
\[
  P^\circ(A):=\bar P(A)+m\one\{x_0\in A\},\qquad A\in\B(\X).
\]
Thus the mass that escaped to \(K\setminus\X\) is moved back to \(x_0\).

First consider the divergence.  If \(f'_\infty=\infty\) and \(D_f(\bar P\|P_0)<\infty\), then \(m=0\), so there is nothing to prove.  If \(f'_\infty<\infty\), write \(\bar P=hP_0+S\), where \(S\perp P_0\) on \(K\).  The escaped mass is part of the singular mass \(S(K)\).  If \(P_0\{x_0\}=0\), moving \(m\) to \(x_0\) preserves the total singular mass and leaves the absolutely continuous density unchanged.  If \(a:=P_0\{x_0\}>0\), the move removes singular mass \(m\) and increases the density at \(x_0\) by \(m/a\).  Convexity implies
\[
  f(u+v)-f(u)\le v f'_\infty,
  \qquad u,v\ge0,
\]
with the convention that the inequality is trivial when the right side is infinite.  Therefore
\[
  a f\!\left(h(x_0)+\frac{m}{a}\right)-a f(h(x_0))
  \le m f'_\infty,
\]
which is no larger than the singular divergence cost that was removed.  Hence \(D_f(P^\circ\|P_0)\le D_f(\bar P\|P_0)\).

For the KL inequality, the restriction of \(\bar P\) to \(\X\) is dominated by \(P^\circ\).  If \(Q\not\ll\bar P\), then \(\KL(Q\|\bar P)=\infty\) and there is nothing to show.  Otherwise, with \(c:=\dd\bar P|_{\X}/\dd P^\circ\le1\),
\[
  \log\frac{\dd Q}{\dd P^\circ}
  =\log\frac{\dd Q}{\dd\bar P|_{\X}}+\log c
  \le \log\frac{\dd Q}{\dd\bar P|_{\X}}
  \qquad Q\text{-a.s.}
\]
Integrating gives \(\KL(Q\|P^\circ)\le\KL(Q\|\bar P)\).
\end{proof}

\subsection{IPM-ball null}\label{subsec:ipm}

\begin{proof}[Proof of Lemma~\ref{lem:ipm-positive}]
If \(\Phi_{\Pcal_{\Fcal}(\rho)}(Q)=0\), then there are \(P_m\in\Pcal_{\Fcal}(\rho)\) with \(\KL(Q\|P_m)\to0\).  Pinsker's inequality gives \(d_{\TV}(Q,P_m)\to0\).  If \(M:=\sup_{f\in\Fcal}\|f\|_\infty\), then
\[
  \gamma_{\Fcal}(Q,P_m)
  \le 2M d_{\TV}(Q,P_m)\to0.
\]
Hence
\[
  \gamma_{\Fcal}(Q,P_0)
  \le \gamma_{\Fcal}(Q,P_m)+\gamma_{\Fcal}(P_m,P_0)
  \le o(1)+\rho.
\]
Letting \(m\to\infty\) gives \(Q\in\Pcal_{\Fcal}(\rho)\), a contradiction.
\end{proof}

Many common IPMs are bounded.  On noncompact sample spaces, positive-radius bounded-IPM balls are often not tight: a small fixed amount of mass can be sent to infinity while keeping the IPM distance bounded.  Weak compactness is then too strong.  The next condition captures the weaker fact needed for lower semicontinuity: any mass that appears on a compactification boundary can be moved back into \(\X\) with arbitrarily small increase in the IPM, and without decreasing the denominator seen by \(Q\).

Recall that for Polish \(\X\), there exists a compact metric space \(K\) containing \(\X\) as a Borel subset, such that the inclusion \(\iota:\X\to K\) is a homeomorphic embedding and weak convergence on \(\X\) implies weak convergence of pushforwards on \(K\).  

\begin{assumption}[IPM compactification and recycling]\label{ass:ipm-recycle}
Fix \(\rho>0\). Assume that each \(f\in\Fcal\) has a bounded continuous extension \(\bar f\in C(K)\).  Moreover, whenever \(\bar P\in\M(K)\) satisfies
\begin{equation}\label{eq:extended-ipm-ball}
  \sup_{f\in\Fcal}\left|\int_K \bar f\dd\bar P-\int_\X f\dd P_0\right|\le \rho,
\end{equation}
then for every \(\eta>0\) there exists \(P_\eta\in\M(\X)\) such that
\begin{enumerate}
\item \(P_\eta(A)\ge \bar P(\iota(A))\) for every Borel \(A\subseteq\X\);
\item \(\gamma_{\Fcal}(P_\eta,P_0)\le \rho+\eta\).
\end{enumerate}
\end{assumption}
\begin{theorem}[Weak lower semicontinuity from recycling]\label{thm:ipm-recycle-lsc}
Under Assumption~\ref{ass:ipm-recycle}, \(Q\mapsto\Phi_{\Pcal_{\Fcal}(\rho)}(Q)\) is weakly lower semicontinuous on \(\M(\X)\). 
\end{theorem}

\begin{proof}
Let \(Q_n\Rightarrow Q\).  As usual, pass to a subsequence attaining the liminf and choose \(P_n\in\Pcal_{\Fcal}(\rho)\) with
\[
  \KL(Q_n\|P_n)\le\Phi_{\Pcal_{\Fcal}(\rho)}(Q_n)+1/n.
\]
View all measures as measures on \(K\).  By compactness of \(K\), a further subsequence satisfies \(P_n\Rightarrow\bar P\in\M(K)\).  Also \(Q_n\Rightarrow Q\) on \(K\).  Since each \(\bar f\) is continuous, \(\bar P\) satisfies \eqref{eq:extended-ipm-ball}.  Joint lower semicontinuity of KL on \(K\) gives
\[
  \KL(Q\|\bar P)\le \liminf_n \KL(Q_n\|P_n),
\]
where \(Q\) is supported on \(\iota(\X)\).

Fix \(\eta>0\), and let \(P_\eta\) be the recycled measure from Lemma~\ref{ass:ipm-recycle}.  Because \(P_\eta\) dominates the restriction of \(\bar P\) to \(\iota(\X)\),
\begin{equation}\label{eq:recycle-kl-ipm}
  \KL(Q\|P_\eta)\le \KL(Q\|\bar P).
\end{equation}
It may only satisfy the enlarged constraint \(\rho+\eta\).  Since \(P_0\) is the center of the ball, define
\[
  \alpha_\eta:=\frac{\eta}{\rho+\eta},
  \qquad
  P_\eta^\circ:=(1-\alpha_\eta)P_\eta+\alpha_\eta P_0.
\]
Then
\[
  \gamma_{\Fcal}(P_\eta^\circ,P_0)
  =(1-\alpha_\eta)\gamma_{\Fcal}(P_\eta,P_0)
  \le (1-\alpha_\eta)(\rho+\eta)=\rho,
\]
so \(P_\eta^\circ\in\Pcal_{\Fcal}(\rho)\).  Also \(P_\eta^\circ\ge(1-\alpha_\eta)P_\eta\), and hence
\[
  \KL(Q\|P_\eta^\circ)
  \le \KL(Q\|P_\eta)-\log(1-\alpha_\eta)
  \le \KL(Q\|\bar P)-\log(1-\alpha_\eta).
\]
Therefore
\[
  \Phi_{\Pcal_{\Fcal}(\rho)}(Q)
  \le \KL(Q\|\bar P)-\log(1-\alpha_\eta)
  \le \liminf_n\Phi_{\Pcal_{\Fcal}(\rho)}(Q_n)-\log(1-\alpha_\eta).
\]
Let \(\eta\downarrow0\).  Since \(\rho>0\), \(\alpha_\eta\downarrow0\), which proves lower semicontinuity.  
\end{proof}

The abstract recycling condition is checkable.  The following lemma clarifies a compactification point.  Polishness alone gives some compact metric embedding of \(\X\), but it does not by itself guarantee that a prescribed IPM class extends continuously to that compactification.  The prescribed class must be compatible with the original topology and small enough to be generated by countably many coordinates if a compact \emph{metric} compactification is desired.

\begin{lemma}[Compactification generated by a separable continuous class]\label[lemma]{lem:separable-compactification}
Let \(\X\) be Polish, and let \(\Fcal\subset C_b(\X)\) be uniformly bounded and separable in the uniform norm \(\|\cdot\|_\infty\).  Then there is a compact metric space \(K\) and a homeomorphic embedding \(\iota:\X\to K\) such that \(\iota(\X)\) is Borel and dense in \(K\), weak convergence on \(\X\) implies weak convergence of the pushforward measures on \(K\), and every \(f\in\Fcal\) has a bounded continuous extension \(\bar f\in C(K)\).
\end{lemma}

\begin{proof}
Because \(\X\) is Polish, there is a countable family \((g_j)_{j\ge1}\subset C_b(\X,[0,1])\) whose evaluation map is a homeomorphic embedding of \(\X\) into the Hilbert cube.  Let \((f_m)_{m\ge1}\) be uniformly dense in \(\Fcal\).  After a harmless affine rescaling, regard the \(f_m\)'s as taking values in a fixed compact interval \([-M,M]\).  Define
\[
  T(x):=\bigl(g_1(x),g_2(x),\ldots; f_1(x),f_2(x),\ldots\bigr)
\]
with values in \([0,1]^{\mathbb N}\times[-M,M]^{\mathbb N}\), and let \(K\) be the closure of \(T(\X)\).  This is compact metric, and \(T\) is a homeomorphic embedding because the \(g_j\)'s already generate the original topology.  A homeomorphic copy of a Polish space inside a metric space is Borel.  The coordinate projections on \(K\) give continuous extensions of the functions \(g_j\) and \(f_m\).

Now fix \(f\in\Fcal\), and choose \(f_{m_
u}\to f\) uniformly on \(\X\).  Let \(\bar f_{m_
u}\) be the coordinate extensions on \(K\).  Since \(T(\X)\) is dense in \(K\), for any two indices \(m,n\),
\[
  \|\bar f_m-\bar f_n\|_{\infty,K}
  =\|f_m-f_n\|_{\infty,\X}.
\]
Thus \((\bar f_{m_
u})\) is uniformly Cauchy on \(K\), and its uniform limit is a continuous extension of \(f\).  Finally, if \(P_n\Rightarrow P\) weakly on \(\X\), then \(T_\#P_n\Rightarrow T_\#P\) weakly on \(K\), because every function in \(C(K)\) restricts to a bounded continuous function on \(\X\).
\end{proof}

\begin{remark}[What is not automatic]\label[remark]{rem:compactification-not-automatic}
The conclusion of \Cref{lem:separable-compactification} is not true for arbitrary IPM classes.  If \(\Fcal\) contains functions that are not continuous for the original topology, then no compactification preserving the original topology can make them continuous: the restriction of a continuous extension would already be continuous on \(\X\).  This is why Kolmogorov--Smirnov indicators require a separate argument.  If \(\Fcal\) contains unbounded functions, then they cannot extend as finite-valued continuous functions to a compact space; this is why Wasserstein balls are handled by tightness and moment arguments rather than by this bounded-IPM compactification.  Finally, if \(\Fcal\subset C_b(\X)\) is very large and nonseparable in \(\|\cdot\|_\infty\), then a compactification extending all functions still exists in the Stone--Cech sense, but it need not be metrizable.  The metric compactification used in the proof  is justified by uniform-norm total boundedness, which implies uniform-norm separability.
\end{remark}

\begin{proof}[Proof of Proposition~\ref{prop:precompact-F}]
Since total boundedness implies separability in \(\|\cdot\|_\infty\), \Cref{lem:separable-compactification} gives a compact metric compactification \(K\) to which all functions in \(\Fcal\) extend continuously, while retaining the original topology of \(\X\). Let \(\bar P\in\M(K)\) satisfy \eqref{eq:extended-ipm-ball}.  Write \(\bar P_X\) for the restriction of \(\bar P\) to \(\X\), and let \(m:=1-\bar P_X(\X)\) be the boundary mass.  If \(m=0\), set \(P_\eta=\bar P_X\).  If \(m>0\), let \(B\) be the normalized boundary restriction of \(\bar P\).  Since \(\X\) is dense in \(K\), there are probabilities \(S_j\) on \(\X\) with \(S_j\Rightarrow B\) weakly on \(K\).  Total boundedness of \(\Fcal\) in the sup norm implies
\[
  \sup_{f\in\Fcal}\left|\int f\dd S_j-\int_K \bar f\dd B\right|\to0.
\]
For sufficiently large \(j\), define \(P_\eta:=\bar P_X+mS_j\).  Then \(P_\eta\) dominates \(\bar P_X\) and \(\gamma_{\Fcal}(P_\eta,P_0)\le \rho+\eta\).  Thus Assumption~\ref{ass:ipm-recycle} holds, and Theorem~\ref{thm:ipm-recycle-lsc} and Lemma~\ref{lem:ipm-positive} prove the claim.
\end{proof}

\begin{lemma}[Weak extension of a bounded Hilbert feature map]\label[lemma]{lem:hilbert-feature-extension}
Let \(\mathcal H\) be a separable Hilbert space and let \(\varphi:\X\to\mathcal H\) be bounded.  Suppose \(K\) is a compact metric space containing \(\X\) as a dense Borel subspace and, for every \(h\in\mathcal H\), the scalar function
\[
  f_h(x):=\langle \varphi(x),h\rangle_{\mathcal H},\qquad x\in\X,
\]
has a continuous extension \(\bar f_h\in C(K)\).  Then there is a unique map \(\bar\varphi:K\to\mathcal H\) such that
\[
  \langle \bar\varphi(z),h\rangle_{\mathcal H}=\bar f_h(z)
  \qquad\text{for all }z\in K,
  h\in\mathcal H.
\]
Moreover \(\sup_{z\in K}\|\bar\varphi(z)\|_{\mathcal H}\le \sup_{x\in\X}\|\varphi(x)\|_{\mathcal H}\), \(\bar\varphi|_\X=\varphi\), and \(\bar\varphi:K\to\mathcal H\) is weakly continuous.  In particular, \(\bar\varphi\) is Borel measurable and Bochner integrable with respect to every probability measure on \(K\).
\end{lemma}

\begin{proof}
Fix \(z\in K\), and define \(L_z(h):=\bar f_h(z)\).  
Since \(\X\) is dense in \(K\), continuous extensions are unique.  Hence, for scalars \(a,b\) and vectors \(h,g\in\mathcal H\), the two continuous functions
\[
  \bar f_{ah+bg}
  \quad\text{and}\quad
  a\bar f_h+b\bar f_g
\]
agree on \(\X\), and therefore agree on all of \(K\).  Thus \(L_z\) is linear.  Also, if \(M:=\sup_x\|\varphi(x)\|_{\mathcal H}\), then
\[
  |L_z(h)|=|\bar f_h(z)|
  \le \sup_{y\in K}|\bar f_h(y)|
  = \sup_{x\in\X}|\langle\varphi(x),h\rangle_{\mathcal H}|
  \le M\|h\|_{\mathcal H}.
\]
The equality of the two suprema uses density of \(\X\) in \(K\) and continuity of \(\bar f_h\).  Hence \(L_z\) is a bounded linear functional on \(\mathcal H\).  By the Riesz representation theorem, there is a unique vector \(\bar\varphi(z)\in\mathcal H\) with \(L_z(h)=\langle\bar\varphi(z),h\rangle\) for all \(h\), and the displayed bound gives \(\|\bar\varphi(z)\|\le M\).  For \(z=x\in\X\), uniqueness gives \(\bar\varphi(x)=\varphi(x)\).

For each fixed \(h\in\mathcal H\), the coordinate map \(z\mapsto\langle\bar\varphi(z),h\rangle=\bar f_h(z)\) is continuous; this is precisely weak continuity.  Since \(\mathcal H\) is separable, weak Borel and norm Borel measurability coincide on bounded subsets, so \(\bar\varphi\) is strongly measurable.  Boundedness then implies Bochner integrability under every finite measure on \(K\).
\end{proof}

\begin{proof}[Proof of Theorem~\ref{prop:mmd-lsc}]
We verify the recycling condition in the Hilbert feature representation.  First, the RKHS unit ball is indeed a subset of \(C_b(\X)\).  Finite linear combinations of kernel sections are continuous because \(k\) is continuous, and if \(h_n\to h\) in \(\mathcal H\), then
\[
  \sup_{x\in\X}|h_n(x)-h(x)|
  =\sup_{x\in\X}|\langle\varphi(x),h_n-h\rangle_{\mathcal H}|
  \le M\|h_n-h\|_{\mathcal H}\to0.
\]
Thus every \(h\in\mathcal H\) is bounded and continuous, and \(x\mapsto\langle\varphi(x),h\rangle_{\mathcal H}=h(x)\) is continuous.  Because \(\mathcal H\) is separable and \(M<\infty\), the RKHS unit ball is separable as a subset of \(C_b(\X)\) in the uniform norm: if \(h_j\to h\) in \(\mathcal H\), then
\[
  \sup_x|h_j(x)-h(x)|
  =\sup_x|\langle\varphi(x),h_j-h\rangle_{\mathcal H}|
  \le M\|h_j-h\|_{\mathcal H}.
\]
Thus \Cref{lem:separable-compactification} can be applied to the RKHS unit ball.  By scaling, every scalar coordinate \(x\mapsto\langle\varphi(x),h\rangle\), \(h\in\mathcal H\), has a continuous extension to the resulting compactification.  Therefore \Cref{lem:hilbert-feature-extension} yields a bounded weakly continuous extension \(\bar\varphi:K\to\mathcal H\).

Let \(\bar P\in\M(K)\) satisfy \(\|\int\bar\varphi\dd\bar P-\mu_{P_0}\|_{\mathcal H}\le\rho\).  The integral is a Bochner integral, justified by \Cref{lem:hilbert-feature-extension}.  Write \(\bar P_X\) for the restriction to \(\X\), let \(m:=1-\bar P_X(\X)\), and define
\[
  v_B:=\int_{K\setminus\X}\bar\varphi\dd\bar P.
\]
If \(m=0\), there is no boundary mass to recycle.  If \(m>0\), then \(v_B/m\) lies in the weak closed convex hull of \(\varphi(\X)\).  Indeed, every \(z\in K\) is the limit of a sequence \(x_n\in\X\); weak continuity gives \(\bar\varphi(x_n)=\varphi(x_n)\to\bar\varphi(z)\) weakly.  Hence \(\bar\varphi(K)\) is contained in the weak closure of \(\varphi(\X)\), and barycenters of probability measures supported on this weak closure lie in its weak closed convex hull.  In a Hilbert space, the weak and norm closures of a convex set agree.  Hence, for every \(\eta>0\), there is a probability \(S_\eta\) supported on finitely many points of \(\X\) such that
\[
  \left\|m\mu_{S_\eta}-v_B\right\|_{\mathcal H}\le\eta.
\]
Set \(P_\eta:=\bar P_X+mS_\eta\).  Then \(P_\eta\) dominates \(\bar P_X\), and
\[
  \left\|\mu_{P_\eta}-\int_K\bar\varphi\dd\bar P\right\|_{\mathcal H}\le\eta.
\]
Therefore \(\MMD_k(P_\eta,P_0)\le\rho+\eta\).  The proof of \Cref{thm:ipm-recycle-lsc} now applies.  Strict positivity follows from \Cref{lem:ipm-positive}, since every RKHS unit-ball function is uniformly bounded by \(M\).
\end{proof}

\begin{proof}[Proof of Theorem~\ref{thm:ks-lsc}]
The zero-radius assertion is immediate.  Assume \(\rho>0\).  Let \(Q_n\Rightarrow Q\), pass to a liminf subsequence, and choose \(P_n\) in the KS ball with \(\KL(Q_n\|P_n)\le\Phi(Q_n)+1/n\).  Compactify \(\RR\) to \(K=[-\infty,\infty]\).  Since \(Q_n\Rightarrow Q\) on \(\RR\), the pushforwards also converge weakly on \(K\).  By compactness, a subsequence of \(P_n\) converges weakly on \(K\) to some \(\bar P\).  KL is weakly lower semicontinuous on \(K\), so
\[
  \KL(Q\|\bar P)\le\liminf_n\KL(Q_n\|P_n).
\]

We first record the inherited KS inequalities.  For finite \(x\), write \(\bar F(x):=\bar P([-\infty,x])\), where the interval includes the compactification point \(-\infty\).  Since \([-\infty,x]\) is closed in \(K\), Portmanteau gives
\[
  \bar F(x)\ge \limsup_n P_n((-
\infty,x])\ge F_{P_0}(x)-\rho.
\]
For \(\delta>0\), the set \([-\infty,x+\delta)\) is open in \(K\), so
\[
  \bar F(x)\le \bar P([-\infty,x+\delta))
  \le \liminf_n P_n((-
\infty,x+\delta))
  \le F_{P_0}(x+\delta)+\rho.
\]
Letting \(\delta\downarrow0\) and using right-continuity of \(F_{P_0}\) yields \(|\bar F(x)-F_{P_0}(x)|\le\rho\) for every finite \(x\).

Let \(a:=\bar P\{-\infty\}\) and \(b:=\bar P\{+\infty\}\).  For \(M>0\), recycle the boundary mass by defining the probability \(P_M\) on \(\RR\) as
\[
  P_M:=\bar P|_{\RR}+a\delta_{-M}+b\delta_M.
\]
This measure dominates \(\bar P|_{\RR}\), so \(\KL(Q\|P_M)\le\KL(Q\|\bar P)\).  We claim that \(d_{\KS}(P_M,P_0)\le\rho+o(1)\) as \(M\to\infty\).  Indeed, for \(-M\le x<M\), the CDF of \(P_M\) equals \(\bar F(x)\), and hence is within \(\rho\) of \(F_{P_0}(x)\).  For \(x<-M\), the CDF of \(P_M\) lies between \(0\) and \(F_{P_M}(-M-)\), while \(F_{P_0}(x)\le F_{P_0}(-M)\to0\).  For \(x\ge M\), the upper deviation is bounded by \(1-F_{P_0}(M)\to0\), and the lower deviation is no worse than that of \(\bar F(x)\), which is bounded by \(\rho\).  Hence the claim holds.

Choose \(M\) so large that \(d_{\KS}(P_M,P_0)\le\rho+\eta\).  As in the proof of \Cref{thm:ipm-recycle-lsc}, mix
\[
  P_M^\circ=(1-\alpha_\eta)P_M+\alpha_\eta P_0,
  \qquad \alpha_\eta=\frac{\eta}{\rho+\eta},
\]
to obtain \(d_{\KS}(P_M^\circ,P_0)\le\rho\) and
\[
  \KL(Q\|P_M^\circ)
  \le\KL(Q\|P_M)-\log(1-\alpha_\eta)
  \le\KL(Q\|\bar P)-\log(1-\alpha_\eta).
\]
Letting \(\eta\downarrow0\) proves lower semicontinuity.  Strict positivity follows from \Cref{lem:ipm-positive}, since KS is generated by functions bounded by one.
\end{proof}

\section{Other non-weakly-compact nulls}\label{app:more-examples}

\subsection{Support-restriction nulls}

Let \(S\subseteq\X\) be closed and define
\[
  \Pcal_S:=\{P\in\M(\X):P(S)=1\}.
\]
If \(S\) is noncompact, then \(\Pcal_S\) is generally not weakly compact: any sequence \(x_n\in S\) with no convergent subsequence yields \(\delta_{x_n}\in\Pcal_S\) with no weakly convergent subsequence.  Nevertheless
\begin{equation}\label{eq:support-Phi}
  \Phi_{\Pcal_S}(Q)
  =
  \begin{cases}
  0, & Q(S)=1,\\
  +\infty, & Q(S)<1.
  \end{cases}
\end{equation}
Indeed, if \(Q(S)=1\), choose \(P=Q\).  If \(Q(S)<1\), then every \(P\in\Pcal_S\) assigns zero mass to \(S^c\), so \(Q\not\ll P\) and \(\KL(Q\|P)=+\infty\).  Since \(S\) is closed, \(\Pcal_S\) is weakly closed by Portmanteau, and the extended-valued function in \eqref{eq:support-Phi} is weakly lower semicontinuous.  This covers support, feasibility, and shape-restriction nulls where observations are required to lie in a known closed subset of a noncompact sample space.

\subsection{Fixed-marginal and observable-only nulls}

Let \(\X=\mathsf Z\times\mathsf Y\), with both spaces Polish, and let \(\pi_Z(z,y)=z\).  Fix \(\mu\in\M(\mathsf Z)\) and define
\[
  \Pcal_\mu:=\{P\in\M(\mathsf Z\times\mathsf Y): (\pi_Z)_\#P=\mu\}.
\]
If \(\mathsf Y\) is noncompact, \(\Pcal_\mu\) is not weakly compact: for any non-tight sequence \((\nu_n)\subset\M(\mathsf Y)\), the product measures \(\mu\otimes\nu_n\) lie in \(\Pcal_\mu\) and are not tight on the product.  Still, the KL projection is explicit:
\begin{equation}\label{eq:fixed-marginal-Phi}
  \Phi_{\Pcal_\mu}(Q)=\KL(Q_Z\|\mu),
\end{equation}
where \(Q_Z=(\pi_Z)_\#Q\).

\begin{proposition}
    In the above setting, \(Q\mapsto\Phi_{\Pcal_\mu}(Q)\) is weakly lower semicontinuous.
\end{proposition}

\begin{proof}
For every \(P\in\Pcal_\mu\), data processing under the projection \(\pi_Z\) gives
\[
  \KL(Q\|P)\ge \KL(Q_Z\|P_Z)=\KL(Q_Z\|\mu).
\]
If \(Q_Z\not\ll\mu\), the right-hand side is infinite, and the claim follows.  Otherwise let \(Q(dy\mid z)\) be a regular conditional distribution under \(Q\), and define
\[
  P^\natural(dz,dy):=\mu(dz)Q(dy\mid z).
\]
Then \(P^\natural\in\Pcal_\mu\), and the chain rule for relative entropy gives
\[
  \KL(Q\|P^\natural)=\KL(Q_Z\|\mu).
\]
Thus \eqref{eq:fixed-marginal-Phi} holds.  Since \(Q\mapsto Q_Z\) is weakly continuous and relative entropy is weakly lower semicontinuous, \(Q\mapsto\Phi_{\Pcal_\mu}(Q)\) is weakly lower semicontinuous.
\end{proof}

The preceding fixed-marginal calculation has a useful observable-only extension, but the reduced optimization must be taken over observable laws that are actually attainable as pushforwards from \(\X\).  This qualification matters when \(T\) is not onto.

\begin{proposition}[Observable-only nulls, with attainable marginals]\label{prop:observable-only-corrected}
Let \(\X\) and \(\mathsf Z\) be Polish spaces, let \(T:\X\to\mathsf Z\) be Borel, and let \(\mathcal A\subseteq\M(\mathsf Z)\).  Define the lifted null
\[
  \Pcal_{T,\mathcal A}
  :=
  \{P\in\M(\X):T_\#P\in\mathcal A\},
\]
and the attainable part of the reduced null
\[
  \mathcal A_T
  :=
  \mathcal A\cap T_\#\M(\X)
  =
  \{\nu\in\mathcal A:\nu=T_\#P\text{ for some }P\in\M(\X)\}.
\]
Then, for every \(Q\in\M(\X)\),
\begin{equation}\label{eq:observable-only-corrected}
  \Phi_{\Pcal_{T,\mathcal A}}(Q)
  =
  \inf_{\nu\in\mathcal A_T}\KL(T_\#Q\|\nu),
\end{equation}
with the convention that the infimum over the empty set is \(+\infty\).  If, in addition, \(T\) is continuous and the reduced functional
\[
  R\mapsto \inf_{\nu\in\mathcal A_T}\KL(R\|\nu)
\]
is weakly lower semicontinuous on \(\M(\mathsf Z)\), then \(Q\mapsto\Phi_{\Pcal_{T,\mathcal A}}(Q)\) is weakly lower semicontinuous.
\end{proposition}

\begin{proof}
Write \(Q_T=T_\#Q\).  For any \(P\in\Pcal_{T,\mathcal A}\), set \(\nu=T_\#P\).  Then \(\nu\in\mathcal A_T\), and data processing gives
\[
  \KL(Q\|P)\ge \KL(Q_T\|\nu).
\]
Taking the infimum over \(P\) gives the lower bound in \eqref{eq:observable-only-corrected}.

Conversely fix \(\nu\in\mathcal A_T\).  If \(\KL(Q_T\|\nu)=+\infty\), there is nothing to prove for this \(\nu\).  Otherwise \(Q_T\ll\nu\).  Since \(\nu\in T_\#\M(\X)\), choose \(\widetilde P\in\M(\X)\) with \(T_\#\widetilde P=\nu\).  Let \(R_\nu(dx\mid z)\) be a regular conditional distribution of \(\widetilde P\) given \(T\), chosen so that \(R_\nu(T^{-1}\{z\}\mid z)=1\) for \(\nu\)-almost every \(z\).  Let \(K_Q(dx\mid z)\) be a regular conditional distribution of \(Q\) given \(T\), chosen so that \(K_Q(T^{-1}\{z\}\mid z)=1\) for \(Q_T\)-almost every \(z\).  Patch these kernels by using \(K_Q\) on a \(Q_T\)-full set where it is supported on the fiber, and using \(R_\nu\) elsewhere.  The resulting kernel \(K(dx\mid z)\) is supported on \(T^{-1}\{z\}\) for \(\nu\)-almost every \(z\), and agrees with \(K_Q\) for \(Q_T\)-almost every \(z\).  Define
\[
  P^{\nu,Q}(dx):=\int K(dx\mid z)\,\nu(dz).
\]
Then \(T_\#P^{\nu,Q}=\nu\), so \(P^{\nu,Q}\in\Pcal_{T,\mathcal A}\), and the chain rule for relative entropy under the disintegration induced by \(T\) gives
\[
  \KL(Q\|P^{\nu,Q})
  =
  \KL(Q_T\|\nu)
  +
  \int \KL\{K_Q(\cdot\mid z)\|K(\cdot\mid z)\}\,Q_T(dz)
  =
  \KL(Q_T\|\nu).
\]
Taking the infimum over \(\nu\in\mathcal A_T\) gives the upper bound and proves \eqref{eq:observable-only-corrected}.

If \(T\) is continuous and \(Q_n\Rightarrow Q\) on \(\X\), then \(T_\#Q_n\Rightarrow T_\#Q\) on \(\mathsf Z\).  The asserted lower semicontinuity follows by composing this weakly continuous pushforward map with the weakly lower semicontinuous reduced functional.
\end{proof}

\subsection{Conditional-kernel models with arbitrary design distribution}

Let \(\X=\mathsf Z\times\mathsf Y\).  Let \(K(dy\mid z)\) be a Markov kernel from \(\mathsf Z\) to \(\mathsf Y\), and define the semiparametric conditional-model null
\[
  \Pcal_K:=\{\nu(dz)K(dy\mid z):\nu\in\M(\mathsf Z)\}.
\]
This is the usual regression-style setup in which the conditional law of \(Y\) given \(Z\) is specified, while the design distribution of \(Z\) is completely unrestricted.  If \(\mathsf Z\) is noncompact, then \(\Pcal_K\) need not be weakly compact, because \(\nu\) may escape to infinity.  Assume that \(K\) is Feller in the following product sense: for every \(h\in C_b(\mathsf Z\times\mathsf Y)\),
\[
  z\mapsto \int h(z,y)K(dy\mid z)
\]
is in \(C_b(\mathsf Z)\).  Then
\begin{equation}\label{eq:conditional-kernel-Phi}
  \Phi_{\Pcal_K}(Q)
  =\KL(Q\|Q_ZK),
\end{equation}
where \(Q_ZK\) denotes the joint law \(Q_Z(dz)K(dy\mid z)\).

\begin{proof}
For any \(\nu\in\M(\mathsf Z)\), the chain rule yields
\[
  \KL(Q\|\nu K)
  =\KL(Q_Z\|\nu)+\KL(Q\|Q_ZK),
\]
with the usual convention that both sides may be infinite.  The infimum over \(\nu\) is therefore attained at \(\nu=Q_Z\), giving \eqref{eq:conditional-kernel-Phi}.  The Feller assumption implies that \(Q_n\Rightarrow Q\) entails \(Q_{n,Z}K\Rightarrow Q_ZK\).  Joint weak lower semicontinuity of KL then gives lower semicontinuity of \(Q\mapsto\KL(Q\|Q_ZK)\).
\end{proof}

The same proof handles a compact parametric family of conditional kernels.  Let \(\Theta\) be compact, and suppose \((z,\theta)\mapsto\int h(z,y)K_\theta(dy\mid z)\) is bounded continuous for every \(h\in C_b(\mathsf Z\times\mathsf Y)\).  For
\[
  \Pcal_\Theta:=\{\nu(dz)K_\theta(dy\mid z):\nu\in\M(\mathsf Z),\theta\in\Theta\},
\]
we have
\[
  \Phi_{\Pcal_\Theta}(Q)
  =\inf_{\theta\in\Theta}\KL(Q\|Q_ZK_\theta),
\]
and this infimum is weakly lower semicontinuous by the usual compact-subsequence argument.  This covers many GLM, conditional density, and Markov-transition nulls with unrestricted initial or design distribution.

\subsection{Independence nulls}

Let \(\X=\mathsf X_1\times\mathsf X_2\), and define
\[
  \Pcal_{\mathrm{ind}}
  :=\{P_1\otimes P_2:P_1\in\M(\mathsf X_1),\ P_2\in\M(\mathsf X_2)\}.
\]
If either factor space is noncompact, \(\Pcal_{\mathrm{ind}}\) is not weakly compact.  Nevertheless,
\begin{equation}\label{eq:independence-Phi}
  \Phi_{\Pcal_{\mathrm{ind}}}(Q)
  =\KL(Q\|Q_1\otimes Q_2)=:I_Q(X_1;X_2),
\end{equation}
where \(Q_i\) is the \(i\)-th marginal of \(Q\).  Indeed, for any product law \(P_1\otimes P_2\), the entropy chain rule gives
\[
  \KL(Q\|P_1\otimes P_2)
  =\KL(Q\|Q_1\otimes Q_2)+\KL(Q_1\|P_1)+\KL(Q_2\|P_2),
\]
so the infimum is attained at \(P_i=Q_i\).  Since marginalization and the product map \((\nu_1,\nu_2)\mapsto \nu_1\otimes\nu_2\) are weakly continuous on Polish spaces, \(Q_n\Rightarrow Q\) implies \(Q_{n,1}\otimes Q_{n,2}\Rightarrow Q_1\otimes Q_2\).  Joint lower semicontinuity of relative entropy proves that mutual information is weakly lower semicontinuous.  Therefore the independence null gives a noncompact but REGROW-compatible null whenever the sample space is noncompact.

\subsection{Finite-group invariance nulls}

Let \(G\) be a finite group acting on \(\X\) by homeomorphisms.  Define the invariant null
\[
  \Pcal_G:=\{P\in\M(\X):g_\#P=P\text{ for all }g\in G\}.
\]
If \(\X\) is noncompact and there are points escaping every compact set, then \(\Pcal_G\) is typically noncompact: the orbit averages \(|G|^{-1}\sum_{g\in G}\delta_{g x_n}\) can escape to infinity.  For \(Q\in\M(\X)\), let
\[
  \bar Q:=\frac1{|G|}\sum_{g\in G} g_\#Q.
\]
Then \(\bar Q\in\Pcal_G\), and
\begin{equation}\label{eq:group-Phi}
  \Phi_{\Pcal_G}(Q)=\KL(Q\|\bar Q).
\end{equation}

\begin{proof}
Fix \(P\in\Pcal_G\).  If \(Q\not\ll P\), then \(\KL(Q\|P)=+\infty\).  Otherwise \(\bar Q\ll P\), and since \(\bar Q\ge |G|^{-1}Q\), also \(Q\ll\bar Q\).  Let \(h=d\bar Q/dP\).  Because both \(\bar Q\) and \(P\) are invariant, \(h\) can be chosen invariant.  Therefore
\[
  \int \log h\,dQ=\int \log h\,d\bar Q.
\]
Using \(dQ/dP=(dQ/d\bar Q)(d\bar Q/dP)\), we get
\[
  \KL(Q\|P)
  =\KL(Q\|\bar Q)+\KL(\bar Q\|P)
  \ge \KL(Q\|\bar Q).
\]
Equality is attained at \(P=\bar Q\).  The map \(Q\mapsto\bar Q\) is weakly continuous because each group action is a homeomorphism, so \eqref{eq:group-Phi} is weakly lower semicontinuous by joint lower semicontinuity of KL.
\end{proof}

This covers sign-flip nulls, label-symmetry nulls, permutation-invariance nulls for fixed finite permutation groups, and other finite randomization symmetries.

\subsection{Bounded moment-inequality and calibration nulls}

Let \(h_1,\ldots,h_m\in C_b(\X)\), and consider the moment-inequality null
\[
  \Pcal_{h,c}:=\left\{P\in\M(\X):\int h_j\,dP\le c_j,\ j=1,\ldots,m\right\}.
\]
Bounded moment inequalities usually do not imply tightness, so \(\Pcal_{h,c}\) may be noncompact on a noncompact \(\X\).  A simple recycling condition gives lower semicontinuity of \(\Phi\).

Assume that \(\X\) is embedded as a Borel subset of a compact metric space \(K\), that each \(h_j\) extends continuously to \(\bar h_j\in C(K)\), and that there exists \(R_\infty\in\M(\X)\) such that
\begin{equation}\label{eq:moment-recycle-condition}
  \int h_j\,dR_\infty
  \le \inf_{z\in K\setminus \X}\bar h_j(z),
  \qquad j=1,\ldots,m,
\end{equation}
with the convention that the condition is void if \(K\setminus\X=\varnothing\).  Then \(Q\mapsto\Phi_{\Pcal_{h,c}}(Q)\) is weakly lower semicontinuous.

\begin{proof}
Let \(Q_n\Rightarrow Q\), and choose nearly optimal \(P_n\in\Pcal_{h,c}\) along a liminf subsequence.  Viewing the measures on \(K\), compactness gives a further subsequence \(P_n\Rightarrow\bar P\in\M(K)\).  Let \(b:=\bar P(K\setminus\X)\).  Since \(\bar h_j\) is continuous,
\(
  \int_K\bar h_j\,d\bar P\le c_j.
\)
Define the recycled probability measure on \(\X\)
\[
  P^\natural(A):=\bar P(A)+bR_\infty(A),
  \qquad A\in\mathcal B(\X),
\]
where \(\bar P(A)\) means \(\bar P\) of the embedded copy of \(A\).  Then
\begin{align*}
  \int h_j\,dP^\natural
  &=\int_\X \bar h_j\,d\bar P+b\int h_j\,dR_\infty\\
  &\le \int_\X \bar h_j\,d\bar P+
       \int_{K\setminus\X}\bar h_j\,d\bar P
   \le c_j,
\end{align*}
so \(P^\natural\in\Pcal_{h,c}\).  Also \(P^\natural\) dominates the restriction of \(\bar P\) to \(\X\), while \(Q\) is supported on \(\X\).  Hence
\[
  \KL(Q\|P^\natural)\le \KL(Q\|\bar P)
  \le \liminf_n\KL(Q_n\|P_n),
\]
where the last inequality is joint lower semicontinuity on the compact space \(K\).  Taking the infimum over null denominators gives the result.
\end{proof}

For example, if the functions \(h_j\) vanish at infinity and there is a point or distribution \(R_\infty\) on which all moments are zero, then \eqref{eq:moment-recycle-condition} holds.  This covers many bounded-instrument GMM, calibration, balance, and fairness-style moment-inequality nulls.  Equality constraints can be handled similarly when the recycling distribution matches the boundary moment vector exactly.

\subsection{KL-coercive noncompact parametric families}

Finally, many ordinary parametric nulls are noncompact but still satisfy \Cref{prop:kl-local-compact}.  Let \(\Theta\) be a sigma-compact metric space, let \(\theta\mapsto P_\theta\in\M(\X)\) be weakly continuous, and set
\[
  \Pcal_\Theta:=\{P_\theta:\theta\in\Theta\}.
\]
Assume the following coercivity condition: whenever \(Q_n\Rightarrow Q\) and \(\theta_n\) leaves every compact subset of \(\Theta\),
\begin{equation}\label{eq:parametric-coercive}
  \KL(Q_n\|P_{\theta_n})\to\infty.
\end{equation}
Then \(\Pcal_\Theta\) is KL-locally compact, and hence \(Q\mapsto\Phi_{\Pcal_\Theta}(Q)\) is weakly lower semicontinuous.

\begin{proof}
Let \(Q_n\Rightarrow Q\), \(P_{\theta_n}\in\Pcal_\Theta\), and \(\sup_n\KL(Q_n\|P_{\theta_n})<\infty\).  By \eqref{eq:parametric-coercive}, \((\theta_n)\) cannot leave every compact subset of \(\Theta\).  Since \(\Theta\) is sigma-compact metric, pass to a subsequence \(\theta_{n_k}\to\theta\) after restricting to a compact set.  Weak continuity gives \(P_{\theta_{n_k}}\Rightarrow P_\theta\in\Pcal_\Theta\).  Thus \(\Pcal_\Theta\) is KL-locally compact, and \Cref{prop:kl-local-compact} applies.
\end{proof}

The preceding coercivity assumption can be verified by showing that escaping parameters put negligible mass on sets on which the limiting alternative has positive mass.  The next statement makes the needed continuity-set condition explicit.

\begin{proposition}[Escape criterion via continuity sets]\label{prop:escape-continuity}
Let \(\theta_n\) leave every compact subset of \(\Theta\).  Suppose that for every \(Q\in\M(\X)\) there exists a Borel set \(A=A(Q,(\theta_n))\) such that
\[
  Q(A)>0,\qquad Q(\partial A)=0,
  \qquad P_{\theta_n}(A)\to0.
\]
Then \(\KL(Q_n\|P_{\theta_n})\to\infty\) for every sequence \(Q_n\Rightarrow Q\).  Consequently, this condition implies the parametric coercivity condition \eqref{eq:parametric-coercive}.
\end{proposition}

\begin{proof}
Since \(A\) is a \(Q\)-continuity set and \(Q_n\Rightarrow Q\), the Portmanteau theorem gives \(Q_n(A)\to Q(A)>0\).  On the other hand \(P_{\theta_n}(A)\to0\).  Applying data processing to the binary map \(x\mapsto\one\{x\in A\}\),
\[
  \KL(Q_n\|P_{\theta_n})
  \ge
  \mathrm{kl}\big(Q_n(A),P_{\theta_n}(A)\big)
  \to\infty.
\]
Here the last divergence tends to infinity because its first argument stays bounded away from zero while its second argument tends to zero.
\end{proof}

\begin{corollary}[Escape from compact sets on proper metric spaces]\label[corollary]{cor:proper-escape}
Assume that \((\X,d)\) is proper, meaning that closed metric balls are compact.  Suppose that, whenever \(\theta_n\) leaves every compact subset of \(\Theta\),
\[
  P_{\theta_n}(C)\to0
  \qquad\text{for every compact }C\subseteq\X.
\]
Then \eqref{eq:parametric-coercive} holds.
\end{corollary}

\begin{proof}
Fix \(Q\in\M(\X)\) and an escaping sequence \(\theta_n\).  Pick \(x_0\in\X\).  The closed balls \(\overline B(x_0,r)\) increase to \(\X\), and the spheres \(\{x:d(x,x_0)=r\}\) have positive \(Q\)-mass for at most countably many values of \(r\).  Hence one can choose \(r>0\) such that
\[
  Q(\overline B(x_0,r))>0,
  \qquad
  Q(\partial \overline B(x_0,r))=0.
\]
Because the metric is proper, \(C:=\overline B(x_0,r)\) is compact, so the assumed escape-from-compact-sets condition gives \(P_{\theta_n}(C)\to0\).  Thus \(C\) satisfies \Cref{prop:escape-continuity}, and the conclusion follows.
\end{proof}

\begin{remark}[Why compact continuity sets are not automatic]\label[remark]{rem:compact-continuity-not-auto}
The compact-continuity-set step is not automatic on an arbitrary Polish space.  In infinite-dimensional Banach or Hilbert spaces, for example, compact sets typically have empty interior.  A compact set with positive \(Q\)-mass then has boundary mass at least that positive mass, so it need not be a \(Q\)-continuity set.  Thus the phrase ``escape from compact sets'' by itself is not a sufficient proof device in full Polish generality.  One should either assume a proper metric, as in \Cref{cor:proper-escape}, or use the conditional formulation in \Cref{prop:escape-continuity}: if one can find a Borel, preferably compact or relatively compact, \(Q\)-continuity set receiving positive \(Q\)-mass and asymptotically zero null mass, then the KL-coercivity proof works.
\end{remark}

Thus Gaussian location families \(\{N(\theta,\Sigma):\theta\in\RR^d\}\), and more generally regular parametric families on \(\RR^d\) whose escaping parameter sequences put vanishing mass on every compact set, satisfy weak lower semicontinuity of \(\Phi\), despite not being weakly compact.  The conclusion also applies to location-scale families along escaping-location or scale-to-infinity subsequences; other boundary regimes should be checked directly through \Cref{prop:escape-continuity} or the KL-coercivity condition \eqref{eq:parametric-coercive}.

\begin{remark}[Finite unions]
If \(\Pcal_1,\ldots,\Pcal_m\) each have weakly lower semicontinuous \(\Phi_i(Q):=\inf_{P\in\Pcal_i}\KL(Q\|P)\), then the finite union \(\Pcal=\bigcup_{i=1}^m\Pcal_i\) also has weakly lower semicontinuous \(\Phi\), because
\[
  \Phi_\Pcal(Q)=\min_{1\le i\le m}\Phi_i(Q),
\]
and a finite minimum of lower semicontinuous functions is lower semicontinuous.  Hence finite model-selection nulls inherit the property from their components.
\end{remark}

\end{document}